\documentclass[12pt]{amsart}
\usepackage{amsthm, amstext, amsmath, amssymb, graphicx, verbatim, fullpage}
\usepackage{booktabs}
\usepackage{hyperref}
\usepackage{ mathrsfs }
\usepackage{algorithm}
\usepackage{xcolor}
\usepackage{versions}
\usepackage{enumerate}
\usepackage{enumitem}
\usepackage{graphicx}
\usepackage{mathtools}
\usepackage{algorithm}
\usepackage{tabularx}
\usepackage{MnSymbol}

\usepackage{tikz-cd}
\usepackage{colonequals}

\usepackage[utf8]{inputenc}
\usepackage[english]{babel}

\newtheorem{thm}{Theorem}[section]
\newtheorem{lemma}[thm]{Lemma} \newtheorem{cor}[thm]{Corollary}
\newtheorem{prop}[thm]{Proposition}
\theoremstyle{definition}
\newtheorem{defn}[thm]{Definition}
\newtheorem{question}[thm]{Question}
\newtheorem{conj}[thm]{Conjecture}
\newtheorem{example}[thm]{Example}

\newtheorem{remark}[thm]{Remark}

\newcommand{\A}{\mathbb{A}}
\newcommand{\Z}{\mathbb Z}

\newcommand{\R}{\mathbb R}
\newcommand{\C}{\mathbb C}
\newcommand{\Q}{\mathbb Q}

\newcommand{\PP}{\mathbb{P}}

\DeclareMathOperator{\charac}{char}

\DeclareMathOperator{\Id}{Id}
\DeclareMathOperator{\Spec}{\text{Spec}}

\DeclareMathOperator{\Sym}{Sym}

\DeclareMathOperator{\ord}{ord}

\newcommand{\abs}[1]{\left|#1\right|}

\newcommand{\mc}{\mathcal}

\DeclareMathOperator{\Top}{top}

\DeclareMathOperator{\dv}{div}

\DeclareMathOperator{\Ind}{Ind}
\DeclareMathOperator{\Exc}{Exc}
\DeclareMathOperator{\adj}{adj}

\DeclareMathOperator{\Num}{Num}

\DeclareMathOperator{\rad}{rad}

\DeclareMathOperator{\Glu}{Glu}

\newcommand\norm[1]{\lVert#1\rVert}

\DeclareMathOperator{\str}{str}
\DeclareMathOperator{\tot}{tot}

\usepackage{pstricks-add}
\usepackage{pst-func}
\usepackage{pst-plot}

\usepackage[foot]{amsaddr}

\usepackage[nocompress]{cite}


\title[The dynamical degree of billiards]{The dynamical degree of billiards in an algebraic curve}
\author{Max Weinreich}
\email{maxhweinreich@gmail.com}
\address{Department of Mathematics, Harvard University, Cambridge, MA 02138\\
 ORCID: 0000-0002-0103-2245}
\date{}                     
\begin{document}
 
\maketitle

\begin{abstract}
We introduce an algebraic formulation of billiards on plane curves over algebraically closed fields, extending Glutsyuk's complex billiards. Given an algebraic plane curve $C$ of degree $d \geq 2$, algebraic billiards is a rational $(d-1)$-to-$(d-1)$ surface correspondence on the space of unit tangent vectors based on $C$. We prove that the dynamical degree of the billiards correspondence is at most an explicit cubic algebraic integer $\rho_d < 2d^2 - d - 3$, depending on the degree $d$ of $C$. As a corollary, for smooth real algebraic curves, the topological entropy of the classical billiards map is at most $\log \rho_d$. We further show that the billiards correspondence satisfies the singularity confinement property and preserves a natural $2$-form. To prove our bounds, we construct a birational model that partially resolves the indeterminacy of algebraic billiards.
\end{abstract}

\section{Introduction} \label{sect_intro}

Mathematical billiards is a dynamical system that models a point particle bouncing around a space \cite{MR2168892, MR3388585}. 
It is often studied as a flow, but in this article, we consider the discrete-time version, the \emph{billiards map}. 
Identify the circle $S^1$ with the space of unit tangent vectors in $\R^2$, and fix a curve $C_0$ in $\R^2$, the \emph{table}. Assume the table $C_0$ is a piecewise-smooth curve bounding a compact, connected region $\Omega$. The billiards map $b_\Omega$ in $\Omega$ is a partially-defined self-map of the subset of $C_0 \times S^1$ consisting of inward-facing unit vectors. Given a vector based at $c$ with direction $q$, we define
$$b_\Omega(c, q) = (c', q'),$$ 
where $c'$ is the first intersection point with $C_0$ of the open ray starting at $c$ of direction $q$, and $q'$ is the reflection of direction $q$ across the tangent line $T_{c'} C_0$, if $c'$ is nonsingular.

Billiards in a complex plane curve $C_0 \subseteq \PP^2_\C$ was introduced by Glutsyuk as a \emph{correspondence}, or meromorphic multi-valued map, on $C_0 \times \PP^1$ \cite{MR3224419}. (The map is multi-valued because intersecting a line with an algebraic curve requires taking a root; see Figure \ref{fig_alg_bill}.) 
Complex billiards has been used to prove a surprising array of results about real billiards that are not evident from real geometry alone \cite{MR4238126, MR4320512, MR3236494, MR3606550, MR4210728, MR3342645}. 

The main goal of this paper is to introduce a purely algebro-geometric formulation of billiards for algebraic plane curves $C_0$ over algebraically closed fields. We prove upper bounds on the dynamical degree,  topological entropy, and orbit growth of algebraic billiards, depending on the degree of the curve. The applications over $\R$ are explained in Section \ref{sect_applications}.

\begin{defn}
    Let $X$ be an irreducible projective variety  over an algebraically closed field $k$. A \emph{dominant rational correspondence} $f = (\Gamma_f, \pi_1, \pi_2)$, also written $f: X \vdash X$, is an effective algebraic cycle $\Gamma_f \neq 0$ in $X \times X$,
    equipped with the projection maps inherited from the product,
    $$\pi_1, \pi_2 : (X \times X)|_{\Gamma_f} \to X.$$
    We require that each component $\Gamma$ of $\Gamma_f$ satisfies $\dim \Gamma = \dim X$, and that on each component $\Gamma$, the maps $\pi_1, \pi_2 : \Gamma \to X$ are dominant and generically finite. We say $f$ is \emph{$m$-to-$n$} if the topological degrees of $\pi_1$ and $\pi_2$ are $n$ and $m$, respectively. Given $x \in X$, the \emph{image} $f(x)$ is the set $f(x) = \pi_2( \pi_1^{-1}(x))$.
\end{defn}

Following the idea of complex billiards, we define algebraic billiards as a rational correspondence. However, our setup is slightly different from previous works on complex billiards such as \cite{MR3224419}. In other works, the role of the ``space of directions'' is played by $\PP^1$, the projectivized tangent space of $\A^2$ at the origin. We instead work with the space of ``unit tangent vectors'', which is naturally a ramified double cover of the projectivized tangent space. This will allow us to mimic certain properties of the classical billiard map more closely.

\begin{defn} \label{def_intro_b}
Let $\Omega \subset \R^2$ be a region such that the Zariski closure of $\partial \Omega \subset \R^2$ in $\PP^2_\C$ is a complex algebraic curve $C$, assumed smooth (thus irreducible) for simplicity. Let $D$ be the complex conic defined as the Zariski closure in $\PP^2_\C$ of 
$$S^1 \colonequals \{ (q_0, q_1) :  q_0^2 + q_1^2 = 1\} \subset \R^2.$$
Let $b_\Omega$ be the real billiards map inside $\Omega$. Its graph is a subset of $(\partial \Omega \times S^1)^2$. The \emph{algebraic billiards correspondence}
$$b: C \times D \vdash C \times D$$
is defined, via its graph $\Gamma_b$, as the Zariski closure of the graph of $b_\Omega$ in $(C \times D)^2$. It is a $(d-1)$-to-$(d-1)$ rational correspondence, where $d = \deg C$ (Proposition \ref{prop_b_basic}).

More generally, given an algebraically closed field $k$ of characteristic not equal to $2$ and a general (for instance, irreducible) curve $C \subset \PP^2_k$ of degree $d \geq 2$, we define an \emph{algebraic billiards correspondence} $b : C \times D \vdash C \times D$ (Definition \ref{def_b}). Note that $b$ depends on the particular embedding of $C$ in $\PP^2_k$; for instance, there is no conjugacy between the billiard correspondences associated to a real circle and a real ellipse.
\end{defn}

The most fundamental algebraic-dynamical invariant of a rational surface correspondence is its \emph{first dynamical degree}, denoted $\lambda_1$, which measures the degree growth of the iterates \cite{MR3667901, MR4048444}.

\begin{defn}[\cite{MR4048444}] \label{def_intro_dd}
     Let $f: X \vdash X$ be a rational correspondence $(\Gamma_f, \pi_1, \pi_2)$ on a smooth projective surface $X$ over an algebraically closed field $k$. Given a divisor $\Delta$, the pushforward $f_* \Delta$ is a divisor with support $f(\Delta)$ that counts components with multiplicity; see Section \ref{sect_prelim}.
     The \emph{degree} of $f$ relative to a fixed ample divisor $\Delta$ is defined by the intersection formula
     $$\deg_\Delta (f) \colonequals f_* \Delta \cdot \Delta.$$
     The \emph{(first) dynamical degree of $f$} is the limit
    \begin{equation} \label{eq_dd_intro}
         \lambda_1(f) \colonequals \lim_{m \to \infty} \deg_\Delta (f^m)^{1/m}.
    \end{equation}
     The limit \eqref{eq_dd_intro} exists, is independent of the choice of divisor $\Delta$, and is a birational conjugacy invariant of $f$.
\end{defn}

\begin{figure}
\includegraphics[width=2.7in]{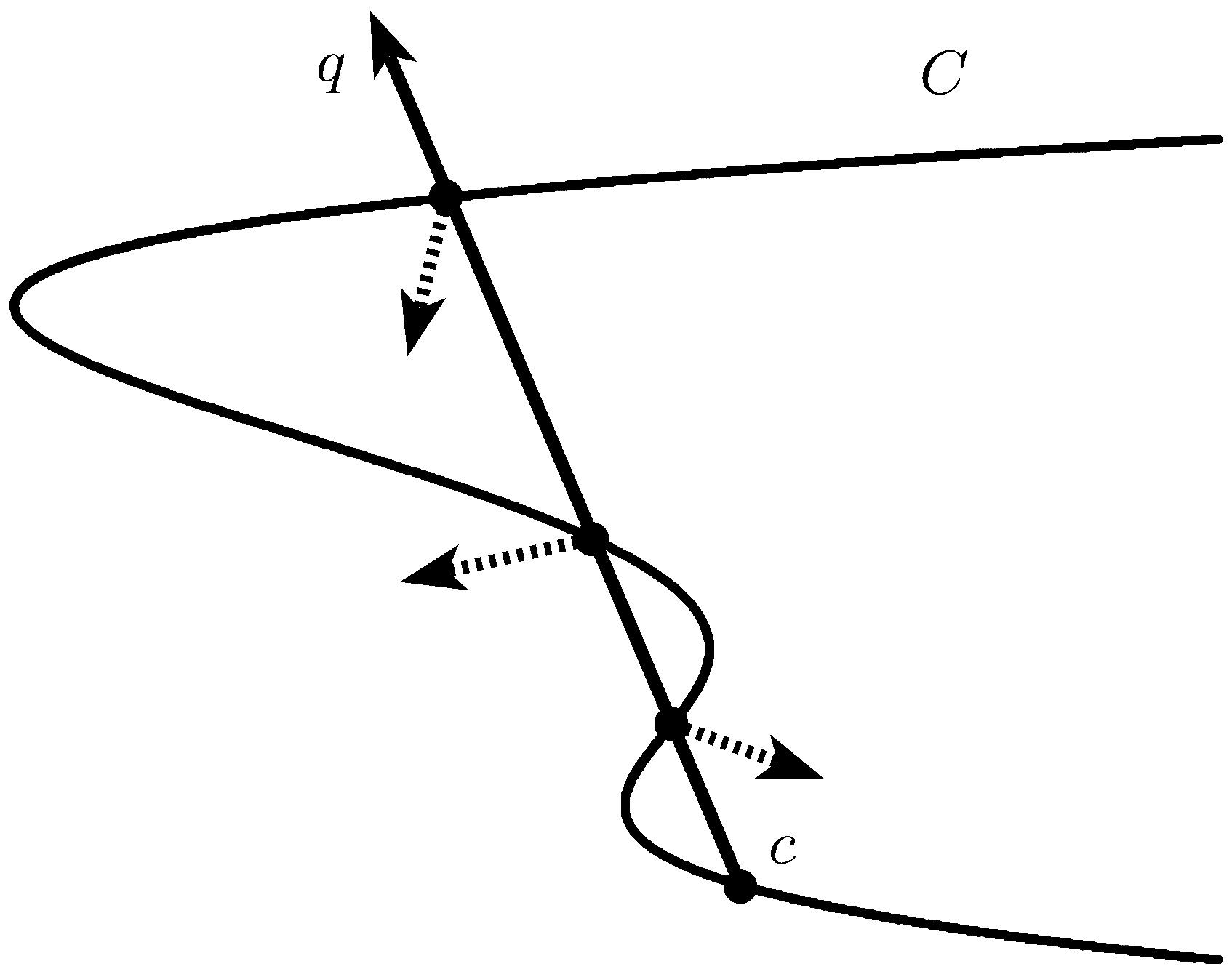}
\caption{Over algebraically closed fields such as $\C$, the billiards map is multivalued. Since non-real intersections of the ray with $C$ cannot be drawn, we show here a real point $c$ of $C$ and a real direction $q$ such that the ray has only real intersection points with $C$. Note that the notion of ``first point of intersection'' of a ray with $C$ does not extend to general $(c, q)$.}
\label{fig_alg_bill}
\end{figure}

Dynamical degrees are very difficult to compute in general; even in the context of rational maps, they can be computed only in low dimensions or in the presence of rich geometric structure such as a group law \cite{MR1867314, MR4276288, MR4363581}. Correspondences beyond single-valued maps present further difficulties. So far, their dynamical degrees have been computed in just two families, monomial correspondences and Hurwitz correspondences \cite{MR4266360, MR3725886, MR4108910}. 

Our first main result is an upper bound on the first dynamical degree $\lambda_1(b)$ of billiards on algebraic curves of a given degree. 

\begin{thm} \label{thm_main}
    Let $k$ be an algebraically closed field of characteristic not equal to $2$. Let $d \geq 2$, and let $C \subset \PP^2_k$ be a smooth algebraic curve of degree $d$. There exists an algebraic integer $\rho_d$ of degree at most $3$, independent of $C$, such that the billiards correspondence
    $b: C \times D \vdash C \times D$
    has first dynamical degree
    \begin{equation}
    \lambda_1(b) \leq \rho_d < 2d^2 - d - 3. 
    \label{eq_main_bound}
    \end{equation}
\end{thm}
Explicitly, the constant $\rho_d$ is the largest root of the polynomial
$$\Phi_d(\lambda) \colonequals \lambda^3 - (2 d^2 - d - 3) \lambda^2 + (2d^2 - 4d + 3) \lambda - (d - 1).$$

One can prove an upper bound on the first dynamical degree $\lambda_1(b)$ of roughly $2d^2$ just by computing the bidegree of $b$; we call this the \emph{cheap bound} (Theorem \ref{thm_cheap_upper_bound}). To get the stronger result of Theorem \ref{thm_main}, we study the orbits of contracted curves.

A \emph{destabilizing orbit} of a correspondence is a contracted curve with an iterate in the indeterminacy locus. A correspondence is \emph{algebraically stable} if it has no destabilizing orbits. In the presence of algebraic stability, the dynamical degree can be computed exactly. 

However, it turns out that billiards has a geometric property called \emph{singularity confinement} that is a kind of opposite to algebraic stability. Singularity confinement is associated with integrable systems and maps with slow degree growth \cite{MR1704282, MR2036429, MR4711236}. Roughly put, a rational correspondence $f$ has the singularity confinement property if the contracted curves of $f$ are not contracted by some higher iterate $f^m$; in other words, these curves ``re-appear'' upon iteration.

\begin{thm} \label{thm_main_sc}
    Let $\charac k \neq 2$, and let $C \subset \PP^2_k$ be a general plane curve of degree $d \geq 2$. The billiards correspondence $b$ on $C \times D$ has the singularity confinement property (Definition \ref{def_sing_conf}).
\end{thm}
Theorem \ref{thm_main_sc} shows that all contracted curves of billiards contribute destabilizing orbits. This leads us to introduce the \emph{modified phase space}, a birational model $P$ for $C \times D$ defined by blowing up all indeterminacy points. The constant $\rho_d$ in Theorem \ref{thm_main} is the spectral radius of the action of $b$ on the group $\Num P \simeq \Z^{2d^2+2}$ of divisor classes on $P$ up to numerical equivalence. The chief difficulty in the proof lies in computing this action.

Since destabilizing orbits are the sole obstruction to computing the dynamical degree precisely, one may ask whether our modified phase space $P$ allows for the computation of $\lambda_1(b)$ rather than just an upper bound. While $P$ resolves all the destabilizing orbits of the correspondence $b$ on $C \times D$, but it often happens that resolving one destabilizing orbit introduces others. We conjecture that this does not happen.

\begin{conj} \label{conj_main}
If $C$ is very general, then the modified phase space $P$ is an algebraically stable model for the billiards correspondence $b$, hence 
$$\lambda_1(b) = \rho_d.$$
\end{conj}

For heuristic evidence towards Conjecture \ref{conj_main}, see Remark \ref{rem_algstab}. In a sequel, we obtain lower bounds on $\lambda_1(b)$ for very general curves under some mild conditions on characteristic \cite{billII}.

We mention one further result that may be useful in future studies of billiards. Real billiards is the prototypical example of a symplectomorphism, that is, a mapping that preserves a nondegenerate alternating $2$-form $\omega$. We generalize the existence of an invariant $2$-form $\omega$ to algebraic billiards and we calculate its divisor (Proposition \ref{prop_div}). 

\begin{thm} \label{thm_main_invariant_form}
    Let $\charac k \neq 2$, and let $b: C \times D \vdash C \times D$ be the billiards correspondence on a curve $C \subset \PP^2_k$ of degree $d \geq 2$. Then $b$ possesses an invariant rational $2$-form $\omega$.
\end{thm}

To explain the origin of the invariant form, note that the domain of the classical billiards map is a subset of the space of unit tangent vectors based on a curve $\partial \Omega$ in $\R^2$. Since reflection is defined with respect to the Euclidean quadratic form, there is a canonical identification of $T \R^2$ with $T^* \R^2$, allowing us to embed the domain as $\partial \Omega \times S^1 \hookrightarrow T^* \R^2$. In the algebraic billiards setting, we similarly have an embedding of a Zariski dense subset of $C \times D$ in $T \A^2$, and via a choice of quadratic form, in $T^* \A^2$. The latter space $T^* \A^2$ has a canonical nondegenerate $2$-form, and the content of Theorem \ref{thm_main_invariant_form} is that the latter form is $b$-invariant.

\begin{remark}
    As mentioned, the dynamical system that Glutsyuk introduced as \emph{complex billiards} in \cite{MR3224419} is a meromorphic correspondence
    $$b_{\Glu} : C \times \PP^1 \vdash C \times \PP^1,$$
    where $\PP^1$ is naturally the projectivized tangent space $\PP(T_0\A^2)$ of $\A^2$ at the origin. In contrast, our space of directions $D$ is naturally the compactification of the conic in $T_0 \A^2$ consisting of tangent vectors $(q_0, q_1)$ such that $q_0^2 + q_1^2 = 1$. While $D$ and $\PP^1$ are abstractly isomorphic, the correspondences $b$ and $b_{\Glu}$ are not conjugate. Rather, there is a $2$-to-$1$ semiconjugacy from $b$ to $b_{\Glu}$ defined by sending a unit tangent vector to its projective class; see Proposition \ref{prop_b_basic}.
    
    By the theory of relative dynamical degrees \cite{MR4048444}, Theorem \ref{thm_main} holds verbatim for $b_{\Glu}$. The singularity confinement property is also preserved by finite semiconjugacy, so Theorem \ref{thm_main_sc} holds for $b_{\Glu}$ as well. However, $b_{\Glu}$ admits no invariant $2$-form (Remark \ref{rem_no_invt_form_on_CP1}). This was our original motivation for introducing algebraic billiards as a correspondence on $C \times D$.
    
    It turns out that $C \times D$ has another advantage, which is that the definitions of the secant and reflection correspondences (that is, intersecting a line with a curve and reflecting it across the tangent) become almost symmetric. This lends an interesting near-duality to most of our definitions and proofs.
\end{remark}

\begin{question} We pose some questions for future research.
\begin{enumerate}
    \item Is the dynamical degree bound of Theorem \ref{thm_main} sharp? 
    \item How does Theorem \ref{thm_main} change for special families of curves, such as reducible curves and singular curves?
    \item Are there any real billiard tables such that the topological entropy over $\R$ and $\C$ agree? There are similarities between our results and those for dynamics on Markoff surfaces, particularly the existence of an invariant $2$-form and a nice model \cite{MR2553877}. On the other hand, the real orbit is just one branch of exponentially many of the complex orbit, so it would be surprising indeed if the two entropies ever agree.
    \item Algebraically stable models for maps with invariant $2$-forms were studied by Diller-Lin \cite{MR3451389}. Can their results be adapted to correspondences such as billiards?
\end{enumerate}
\end{question}

\subsection{Applications to real billiards} \label{sect_applications}
Dynamical degrees play a central role in the arithmetic, ergodic, and entropic properties of multivariate dynamical systems \cite{MR2180409,MR4048444,MR4007163}. For the billiards correspondence, the first dynamical degree $\lambda_1$ is an upper bound for topological entropy, via a comparison theorem of Dinh-Sibony \cite{MR2391122}. Thus our techniques lead to upper bounds on the topological entropy of billiards in smooth algebraic plane curves.

\begin{cor} \label{cor_top}
Let $\Omega \subset \R^2$ be a region such that the Zariski closure of $\partial \Omega$ in $\PP^2_\C$ is an smooth algebraic curve $C$ of degree $d \geq 2$. Then the billiards map $b_\Omega$ inside $\Omega$ satisfies
$$h_{\Top}(b_\Omega) \leq \log \rho_d < \log (2d^2 - d - 3).$$
\end{cor}

For instance, let $b_\Omega$ be the billiards map in a bounded real component of an elliptic curve in $\R^2$. The cheap bound (Section \ref{sect_cheap}) gives
$$h_{\Top}(b_\Omega) \leq \log 16$$
and our main result gives
$$h_{\Top}(b_\Omega) < \log 11.2.$$

Strangely, the entropy of the time-one map of the billiard flow in a real algebraic curve can nevertheless be infinite \cite{MR2201936}. 

The relationship between the geometry of the boundary curve $C$ and the entropy of the billiards map $b_\Omega$ is very subtle \cite{MR3754521, MR1468104, MR896765, MR872698, MR4213303, MR4632917, MR2201936, MR2609011}. In an ellipse, the billiards map has topological entropy $h_{\Top}(b_\Omega) = 0$; the celebrated Birkhoff conjecture implies that this property characterizes ellipses among smooth, strictly convex Jordan curves. See the survey \cite{MR3388585} and the instructive simulation \cite{butterflyeffect}. On the other hand, even tiny deviations from an ellipse can have ``chaos'', or positive entropy. The question of calculating topological entropy of billiards in a given curve is a vast enrichment of the Birkhoff conjecture.

Upper bounds on entropy were previously known only for a few simple tables such as polygons ($h_{\Top} = 0$), ellipses ($h_{\Top} = 0$), and stadia ($h_{\Top} < \log 3.5$) \cite{MR3754521, MR1468104, MR3388585, MR4632917}. In degree $d = 2$, Theorem \ref{thm_main} and Corollary \ref{cor_top} give a new proof of the complete integrability of the billiard in an ellipse, from which one may obtain Poncelet's porism (Corollary \ref{cor_poncelet}). The other well-studied tables, polygons and stadia, are not smooth, so Corollary \ref{cor_top} does not apply.

Another application of Theorem \ref{thm_main} concerns orbit growth in real algebraic billiards, following the study of elliptic billiards by Corvaja-Zannier \cite{MR4668525}. Fix points $c, c' \in \partial \Omega$ and an integer $m > 0$. How many ways can a ball fired from $c$ hit a target at $c'$ in $m$ bounces?

\begin{cor} \label{cor_orbit_growth}
Let $\Omega \subset \R^2$ be a region such that the Zariski closure of $\partial \Omega$ in $\PP^2_\C$ is a smooth algebraic curve $C$ of degree $d \geq 2$. Let $b_\Omega$ be the real billiard map in $\Omega$, and let $c, c' \in \partial \Omega$. Let 
$$a_m(c, c') = \quad \# \{ q \in S^1 : b_\Omega^m (c, q) = (c', q') \quad \textrm{ for some } q' \in S^1 \}. $$
Assume that $a_m(c, c') < \infty$ for all $m$. Then, if $\rho_d < 2d^2 - d - 3$ is the constant of Theorem \ref{thm_main}, we have that for all $\epsilon > 0$,
$$ a_m(c, c') = O(( \rho_d + \epsilon)^m). $$
\end{cor}

Indeed, basic intersection theory shows that the complex intersections of $c' \times D$ with $b^m(c \times D)$ are counted by the dynamical degree. The corollary follows, since the real intersections are among the complex ones. Of course, it seems likely that vanishingly few of the complex intersections are in fact real, but this seems difficult to prove. The finiteness condition excludes exceptional cases in which the curves $c' \times D$ and $b^m(c \times D)$ share a component. For a reducible example, take $m = 2$ and $c = c'$ the center of a semicircle.

In degree $d = 2$, Corollary \ref{cor_orbit_growth} only shows that $a_m$ grows subexponentially, but in fact much more is known \cite{MR4668525}. Indeed, when $\partial \Omega$ is an ellipse, billiards is completely integrable, and $a_m(c, c')$ grows linearly in $m$ with an explicit constant in terms of elliptic integrals. The number of complex directions bouncing $c$ to $c'$ grows quadratically.

\subsection{Sketch of proofs}

Our proofs are written in the language of algebraic dynamics and require no prior knowledge about billiards.

We decompose billiards as a composition of two correspondences $s$ and $r$. The \emph{secant correspondence} $s$ intersects a line of given base point $c$ and direction $q$ with the table $C$, and the \emph{reflection correspondence} $r$ reflects the direction $q$ across the tangent line at $c$.

We prove Theorem \ref{thm_main_invariant_form} first, by direct computations with differential forms. Our computation was inspired by a symplectic reduction argument in pseudo-Riemannian billiards by Khesin-Tabachnikov \cite{MR2518642}, which in some sense goes back to Melrose \cite{MR436225, MR614395}.

The proofs of Theorem \ref{thm_main} and Theorem \ref{thm_main_sc} rely on a careful analysis of the behavior of the billiards correspondence at points of indeterminacy. We dub these indeterminacies \emph{scratch points}. (In the real-life game of pool, a \emph{scratch} is a shot that causes the cue ball to enter a pocket, ending its orbit.) Scratch points come in two flavors. \emph{Scratch points at infinity} correspond to points of $C$ on the line at infinity of $\PP^2$; these are indeterminacy points of the secant correspondence $s$. \emph{Isotropic scratch points} correspond to points of $C$ where the tangent line has slope $\pm i$; these are indeterminacy points of the reflection correspondence $r$. The word \emph{isotropic} appears because the slopes $\pm i$ are self-normal directions for the natural complex metric $dx_0^2 + dx_1^2$.

The key technical tool we introduce is the \emph{modified phase space} $P$, defined as the blowup of the directed phase space $C \times D$ at every scratch point (Definition \ref{def_mod_phase_space}). We show that billiards has much less indeterminacy on the birational model $P$ in a series of four lemmas that describe the local behavior of $r$ and $s$ at blown-up scratch points (Section \ref{sect_local}). Note that blowing up indeterminacy points of a ``typical'' map or correspondence just creates new ones, but in our case $P$ is genuinely a better space of initial conditions than $C \times D$; in fact, we expect that $P$ is an algebraically stable model when $C$ is very general.

The proof of Theorem \ref{thm_main} proceeds by computing the action of the billiards correspondence $b$ on the cohomology of the modified phase space $P$, via the pushforward linear map $b_*$. To get a finite-dimensional vector space, we work with the group $\Num^1 P$ of divisors up to numerical equivalence, and we compute the spectral radius of $b_*$. Here, there is another surprise: although $\Num^1 C \times D \cong \Z^{2+2d^2}$, where $d = \deg C$, we nonetheless find that the leading eigenvalue of $b_*$ is an algebraic integer of degree at most $3$.

Finally, the proof of singularity confinement (Theorem \ref{thm_main_sc}) is another application of the local blowup lemmas in Section \ref{sect_local}; these allow us to compute the images of each contracted curve by $b^2$.

\subsection*{Road map}

Section \ref{sect_prelim} contains preliminaries on correspondences and dynamical degrees. Section \ref{sect_basic} defines algebraic billiards as a correspondence over any algebraically closed field of characteristic not equal to $2$. Section \ref{sect_invariant_form} proves Theorem \ref{thm_main_invariant_form}. Section \ref{sect_cheap} proves the ``cheap'' version of Theorem \ref{thm_main} via an argument that is short but not sharp. Section \ref{sect_local} shows how to resolve contracted curves and indeterminacy locally by blowing up scratch points. Section \ref{sect_global_bound} proves Theorem \ref{thm_main}, Corollary \ref{cor_top}, and Corollary \ref{cor_orbit_growth}. Section \ref{sect_sc} introduces and proves the singularity confinement property of billiards, Theorem \ref{thm_main_sc}.

\subsection*{Acknowledgments}

Thanks to Richard Birkett, Laura DeMarco, Nguyen-Bac Dang, Jeffrey Diller, Alexey Glutsyuk, Hannah Larson, Curt McMullen, Rohini Ramadas, Joseph Silverman, Serge Tabachnikov, Jit Wu Yap, and Umberto Zannier for helpful conversations related to this work. The idea to study dynamical degrees of billiards came while reading Cannas da Silva's lectures on symplectic geometry \cite{MR1853077}. The author was supported by a National Science Foundation Mathematical Sciences Postdoctoral Research Fellowship under Grant No. 2202752.

\section{Background on Dynamics of Correspondences} \label{sect_prelim}

\subsection{Correspondences}
Essential background on algebraic cycles and intersection theory may be found in \cite{MR1644323, MR3617981}.

\begin{defn} \label{def_rat_corr}
    Given irreducible varieties $X$ and $Y$, a \emph{rational correspondence} $f: X \vdash Y$, also written
    $$f = (\Gamma_f, X, Y) = (\Gamma_f, \pi_1, \pi_2),$$
    is an effective algebraic cycle in $X \times Y$ of the form
$$\Gamma_f = \sum_{i = 1}^\nu m_i [\Gamma_i] \neq 0,$$
such that each summand satisfies $\dim \Gamma_i = \dim X$. Our setting is that of \emph{surface correspondences}; that is, we always have 
$$\dim X = \dim Y = 2.$$
We denote the projections $\Pi_1 : X \times Y \to X$ and $\Pi_2 : X \times Y \to Y$; their restrictions to $\Gamma_f$ are denoted $\pi_1$, $\pi_2$.

A rational correspondence of irreducible varieties is \emph{dominant} if the projection maps $\pi_1$ and $\pi_2$ restrict to dominant, generically finite morphisms on each $\Gamma_i$. 
We always assume $\Gamma_f \neq 0$. We say that $f$ is \emph{$m$-to-$n$}, where $m = \deg \pi_2$ and $n = \deg \pi_1$. 

We frequently view $\Gamma_f$ itself as a variety in situations where multiplicities may be ignored.

\begin{remark}
    Let $\Sym^n Y$ denote the $n$-th symmetric product of $Y$. A rational map $\phi: X \dashrightarrow \Sym^{n} Y$ of degree $m$ defines an $m$-to-$n$ rational correspondence $\Gamma_\phi$, by taking the Zariski closure of the locus of all pairs $(x, y) \in X \times Y$ such that $\phi$ is defined at $x$ and $y \in \phi(x)$. When $n = 1$, this recovers the usual notion of the graph of a rational map.
\end{remark}

The \emph{indeterminacy locus} $\Ind f \subset X$ of $f : X \vdash Y$ is the set 
$$\Ind f \colonequals \{x \in X: \pi_1^{-1}(x) \text{ is not finite} \}.$$

The \emph{exceptional locus} $\Exc f \subset Y$ of $f: X \vdash Y$ is the set
$$\Exc f \colonequals \{y \in Y: \pi_2^{-1}(y) \text{ is not finite} \}.$$

Given rational correspondences $f, g: X \vdash Y$, we define $f + g$ to be the reducible correspondence with graph $\Gamma_f + \Gamma_g$.
\end{defn}

\begin{remark} \label{rem_reduc}
The definition of dominant correspondence may be extended to domains that are reducible varieties as follows; see \cite[Section 3.1]{MR4048444}. Let $X$ and $Y$ be varieties. Let $X_1, \dots, X_\gamma$ be the irreducible components of $X$, and let $Y_1, \dots, Y_\delta$ be the irreducible components of $Y$. A \emph{correspondence} $f : X \vdash Y$ is a collection of correspondences $f_{i,j}: X_i \vdash Y_j$, where $1 \leq i \leq \gamma$ and $1 \leq j \leq \delta$. 

The correspondence $f$ is \emph{dominant} if and only if each $f_{i, j}$ is dominant, the union of the domains of definition of the $f_{i, j}$ is dense in $X$, and the union of the ranges of the $f_{i, j}$ is dense in $Y$.

Then one can define compositions, and all of Truong's results on dynamical degrees then hold; see \cite[Section 6.2]{MR4048444}. Our main focus in this paper is on correspondences on irreducible varieties, but we point out when reducibility comes into play.
\end{remark}

\begin{defn} \label{def_adjoint}
    Let $f: X \vdash Y$ be a correspondence. The \emph{adjoint} of $f$, denoted $f^{\adj}$, is the correspondence $Y \vdash X$ defined by the image of $\Gamma_f \subset X \times Y$ in $Y \times X$ under the map $(x, y) \mapsto (y,x)$. A correspondence $f : X \vdash X$ is \emph{symmetric} or \emph{self-adjoint} if $f = f^{\adj}$.
\end{defn}

\begin{defn} \label{def_composite}
Given $f: X \vdash Y$ and $g: Y \vdash Z$, the \emph{composite} $g \circ f: X \vdash Z$ is defined as follows, generalizing the case of rational maps. First, suppose that the graphs $\Gamma_f$ and $\Gamma_g$ are reduced and irreducible, that is, each consists of the class of a single variety with multiplicity $1$. We may then treat the graphs $\Gamma_f$ and $\Gamma_g$ as varieties. We name the projections
$$\pi_1 : \Gamma_f \to X,$$
$$\pi_2 : \Gamma_f \to Y,$$
$$\pi_3 : \Gamma_g \to Y,$$
$$\pi_4 : \Gamma_g \to Z.$$
Let $V = Y \smallsetminus \Ind g$, and let $U \subset X$ be $U = f^{-1}(V) \smallsetminus \Ind f$. These restrictions give us open subsets $\Gamma_f^{U \times V} \colonequals \Gamma_f|_{U \times V}$ and $\Gamma_g^{V \times Z} \colonequals \Gamma_g|_{V \times Z}$. We then form the fiber product as a scheme inside $U \times V \times Z$:
$$ \Gamma_{g \circ f}^{U \times V \times Z} \colonequals \Gamma_f^{U \times V} \times_V \Gamma_g^{V \times Z}.$$
Let $\Gamma_{g \circ f}^{U \times Z}$ be the image of $\Gamma_{g \circ f}^{U \times V \times Z}$ by the projection 
$$ U \times V \times Z \to U \times Z.$$
Then, the scheme $\Gamma_{g \circ f}^{U \times Z}$ may be viewed as an algebraic cycle of dimension $n$, possibly reducible; so write
$$ \Gamma^{U \times Z}_{g \circ f} = \sum m_i [\Gamma_i] .$$
For each $i$, let $\bar{\Gamma}_i$ be the Zariski closure of $\Gamma_i$ in $X \times Z$.
We define $g \circ f$ as the correspondence $X \vdash Z$ with graph
$$ \Gamma_{g \circ f} \colonequals \Gamma_{g \circ f}^{X \times Z} \colonequals \sum_i m_i [\bar{\Gamma}_i].$$

To define composites of correspondences with reduced or reducible graphs, we extend the definition by the rules 
$$\Gamma_{(f + g) \circ h} = \Gamma_{f \circ h} + \Gamma_{g \circ h}$$
and 
$$\Gamma_{f \circ (g + h)} 
 = \Gamma_{f \circ g} + \Gamma_{f \circ h}.$$
\end{defn}

 \begin{remark} \label{rem_composite_nat_arrows}
 Let $f: X \vdash Y$ and $g: Y \vdash Z$. The following diagram is helpful in understanding the construction of the composite. The first is the formation of the fiber product, the second applies the universal property of the product $U \times Z$, and the third takes the Zariski closure in $X \times Z$. The dashed arrows denote rational maps.
 \begin{center}
\begin{tikzcd}[row sep=small, column sep=small]
     &        & \Gamma^{U \times V \times Z}_{g \circ f} \arrow[dl] \arrow[dr] &   & \\
     & \Gamma^{U \times V}_f \arrow[dl] \arrow[dr] & & \Gamma^{V \times Z}_g \arrow[dl] \arrow[dr] & \\
U &    & V &     & Z
\end{tikzcd} \hspace{1.3in}
\begin{tikzcd}[row sep=small, column sep=small]
     &         \Gamma_{g \circ f}^{U \times V \times Z} \arrow[d,->>] \arrow[ddr, bend left] \arrow[ddl, bend right] \\
     & \Gamma_{g \circ f}^{U \times Z} \arrow[ld] \arrow[rd] \\
     U & & Z
\end{tikzcd}
\begin{tikzcd}[row sep=small, column sep=small]
     &        & \Gamma_{g \circ f} \arrow[dl, dashed] \arrow[ddll,bend right=40] \arrow[dr, dashed] \arrow[ddrr, bend left=40] &   & \\
     & \Gamma_f \arrow[dl] \arrow[dr] & & \Gamma_g \arrow[dl] \arrow[dr] & \\
X &    & Y &     & Z
\end{tikzcd}
\end{center}
The composition of rational maps $\Gamma_{g \circ f} \dashrightarrow \Gamma_f \to X$ agrees with the morphism $\Gamma_{g \circ f} \to \Gamma_X$, and $\Gamma_{g \circ f} \dashrightarrow \Gamma_g \to Z$ agrees with the morphism $\Gamma_g \to Z$. The two rational maps $\Gamma_{g \circ f} \dashrightarrow \Gamma_f \to Y$ and $\Gamma_{g \circ f} \dashrightarrow \Gamma_g \to Y$ agree.
 \end{remark}

The following definition is not standard, but is very useful.
\begin{defn}[Contractions and expansions] \label{def_contraction}
 Assume $X, Y$ are surfaces. A \emph{contraction} of $f: X \vdash Y$ is an integral curve $V$ in $\Gamma_f$ such that $\pi_1(V)$ is a curve and $\pi_2(V)$ is a point. 
 
 An \emph{expansion} of $f: X \vdash Y$ is an integral curve $W$ in $\Gamma_f$ such that $\pi_1(W)$ is a point and $\pi_2(W)$ is a curve.
 \end{defn}

 \begin{example}
     The following claims are all immediate from the definitions. Say $f : X \vdash Y$ is $m$-to-$n$.
 \begin{enumerate}
     \item If $V$ is a contraction, then $\pi_2(V) \in \Exc f$.
     \item If $V$ is an expansion, then $\pi_1(V) \in \Ind f$.
     \item The adjoint $f^{\adj}$ is $n$-to-$m$.
     \item We have $\Ind f^{\adj} = \Exc f$ and $\Exc f^{\adj} = \Ind f$. We may identify contractions of $f$ with expansions of $f^{\adj}$, and expansions of $f$ with contractions of $f^{\adj}$.
     \item If $f$ is a birational map, then $f^{\adj}$ is its inverse.
 \end{enumerate}
 \end{example}

 \begin{example} \label{ex_blowup}
     The reader can build intuition by verifying the following claims. Let $\tau: P \vdash \PP^2$ be the point blowup at $p$, with exceptional divisor $E$. Let $\sigma: \PP^2 \vdash P$ be the adjoint of $\tau$.
     \begin{enumerate}
         \item The only contraction of $\tau$ is $E \times p$, and $\tau$ has no expansions.
         \item The graph $\Gamma_{\sigma \circ \tau}$ is the diagonal embedding of $P$ in $P \times P$, so $\sigma \circ \tau = \Id_P$.
         \item The fiber product $\Gamma_\tau \times_{\PP^2} \Gamma_\sigma$ properly contains $\Gamma_{\sigma \circ \tau}$.
         \item The natural rational maps $\Gamma_{\sigma \circ \tau} \dashrightarrow \Gamma_{\sigma}$ and $\Gamma_{\sigma \circ \tau} \dashrightarrow \Gamma_{\tau}$ are regular maps, and the inverse image of $E \times p \subset \Gamma_\sigma$ in $\Gamma_{\sigma \circ \tau}$ is the diagonal embedding of $E$ in $P \times P$.
         \item The graph $\Gamma_{\tau \circ \sigma}$ is the diagonal embedding of $\PP^2$ in $\PP^2 \times \PP^2$, so $\tau \circ \sigma = \Id_{\PP^2}$.
         \item We have an equality $\Gamma_\sigma \times_{P} \Gamma_{\tau} = \Gamma_{\tau \circ \sigma}$.
         \item The natural rational maps $\Gamma_{\tau \circ \sigma} \dashrightarrow \Gamma_{\tau}$ and $\Gamma_{\tau \circ \sigma} \dashrightarrow \Gamma_{\sigma}$ have indeterminacy.
     \end{enumerate}
     This example illustrates a key motivation for restricting to the subsets $U$ and $V$ when forming the fiber product in Definition \ref{def_composite}. To have a category, the composite should be a dominant rational correspondence. The ``naive'' fiber product $\Gamma_f \times_Y \Gamma_g$ may contain irreducible components like $E \times E$ on which the projections are not dominant.
 \end{example}

Modifying the domain of a correspondence is often useful:
\begin{defn} \label{defn_tot_transform}
    Given a correspondence $f : X \vdash X$, and a blowup $\tau: P \to X$ over a point $p$, the \emph{total transform} is the scheme-theoretic pullback of $\Gamma_f$ to $P \times P$. The \emph{strict transform} $f_{\str} : P \vdash P$ is the rational correspondence
$$f_{\str} \colonequals \tau^{-1} \circ f \circ \tau.$$
\end{defn}

 \subsection{Pushforward and pullback}
This section presents the standard definitions of pushforward and pullback of divisor classes by surface correspondences of smooth, irreducible varieties.
\begin{defn}
    Given a rational correspondence $f: X \vdash X$, where $f = (\Gamma_f, \pi_1, \pi_2)$, and a set $\Sigma \subset X$, the \emph{(total) image} of $\Sigma$ by $f$ is the set
$$f(\Sigma) = \pi_2(\pi_1^{-1}(\Sigma)).$$
\end{defn}

If a correspondence has expansions or contractions, then taking set-theoretic images can cause the dimension of a set to change. The pushforward, which we define in this section, is a kind of image operation that respects dimension and multiplicity.

In this paper, a \emph{multiset} is an effective $0$-cycle. When working with multisets, we use the symbol $\smallsetminus$ to denote multiset difference, taking multiplicities into account.

A projective variety $X$ has an associated Chow ring $A^\ast (X)$ and Chow groups $A_\ast (X)$. When $X$ is smooth, we may identify $A^\ast(X)$ with $A_\ast(X)$. The elements of $A^\ast(X)$ are algebraic cycles up to rational equivalence, graded by cycle dimension. Addition is formal, and multiplication is defined via the intersection product $\cdot$. In this paper, we need only consider the $1$-graded parts $A^1(X)$ and $A_1(X)$.

\begin{defn}[{\cite[Section 1.3]{MR3617981}
}]\label{def_reg_pb_pf} 
    Let $u : V \to W$ be a generically finite and dominant morphism of smooth projective surfaces, and let $[\Sigma]$ be a prime cycle of $V$. Let the generic points of $\Sigma$ and $u(\Sigma)$ be denoted $\theta, \eta$ respectively. Since the surfaces $V$ and $W$ are projective, the map $u$ is proper. The \emph{proper pushforward} $u_* [\Sigma]$ is defined by
\begin{equation}
    u_* [\Sigma] \colonequals \begin{cases}
        0, & \dim u(\Sigma) < \dim \Sigma,\\
        [ K(\theta) : K(\eta) ] \; [u(\Sigma)], & \text{otherwise}. 
    \end{cases}
\end{equation}
This map is well-defined on rational equivalence classes, so it descends to the prime elements of the Chow group; then it extends by linearity to define a graded group homomorphism
$$u_* : A_*(V) \to A_*(W).$$

There is also a \emph{pullback} homomorphism $u_* : A_*(W) \to A_*(V)$. For our purposes, we only need to know its values in simple situations, as follows. Let $\Sigma$ be a prime $1$-cycle of $W$, that is, a prime Weil divisor. Say $f^{-1}(\Sigma)$ is generically reduced; this happens, for instance, if the inverse image of a general point of $\Sigma$ consists of $d$ distinct points, where $d = \deg u$. Then
$$u^* [\Sigma] = [u^{-1} (\Sigma)].$$
\end{defn}

\begin{defn} \label{def_corr_push}
A dominant correspondence $f: X \vdash Y$ of smooth projective surfaces, defined by $f = (\Gamma_f, \pi_1, \pi_2)$, has an associated \emph{pushforward} homomorphism 
$$f_*: A_*(X) \to A_*(Y),$$
and a \emph{pullback} homomorphism 
$$f^*: A_*(Y) \to A_*(X).$$
If $\Gamma_f$ is smooth and irreducible, these are defined respectively by
$$f_* \colonequals (\pi_2)_* \circ (\pi_1)^*$$
$$f^* \colonequals (\pi_1)_* \circ (\pi_2)^*.$$
The case of reducible $\Gamma_f$ is not needed for our purposes.

If $\Gamma_f$ is not smooth, then let $\gamma: \Gamma'_f \to \Gamma_f$ be a minimal resolution of singularities by blowups, let $\pi'_1 = \pi_1 \circ \gamma$, and let $\pi'_2 = \pi_2 \circ \gamma$. Then define
$$f_* \colonequals (\pi'_2)_* \circ (\pi'_1)^*,$$
$$f^* \colonequals (\pi'_1)_* \circ (\pi'_2)^*.$$
These definitions are independent of the desingularization $\Gamma'_f$; see e.g. \cite[Lemma 3.2]{MR3342248} for a proof stated for $\C$ but that holds over any algebraically closed field.
\end{defn}

A Chow class $\alpha \in A^*(X)$ is \emph{numerically equivalent to $0$} if, for all $\beta \in A^*(X)$, we have $\alpha \cdot \beta = 0$. Two cycles are \emph{numerically equivalent} if their difference is numerically equivalent to $0$. We denote the group of cycle classes up to numerical equivalence by $\Num X$. It is a finitely generated free abelian group.

We let the $1$-graded part of $\Num X$ be denoted $\Num^1 X$. Pushforward and pullback descend to well-defined group homomorphisms
$$f_* :  \Num^1 X \to \Num^1 Y,$$
$$f^* :  \Num^1 Y \to \Num^1 X.$$

\begin{lemma}
Let $f : X \vdash Y$ be a dominant correspondence of smooth, irreducible projective surfaces, let $\Delta \in \Num^1 X$, and let $\Delta' \in \Num^1 Y$. Let $\Pi_1 : X \times Y \to X$ and $\Pi_2 : X \times Y \to Y$ be the projections. Let $\iota_f(\Delta, \Delta') \colonequals f_* \Delta \cdot \Delta'$. Then
\begin{equation} \label{eq_pfwd_ix}
    \iota_f(\Delta, \Delta') = \Gamma_f \cdot \Pi_1^* \Delta \cdot \Pi_2^* \Delta'.
\end{equation}
\end{lemma}

\begin{proof}
An equivalent definition of pushforward by a irreducible, reduced correspondence $f$ is
$$f_* \Delta = (\Pi_2)_* (\Gamma_f \cdot \Pi_1^* \Delta),$$
see e.g. \cite{MR3342248} (the equivalence argument in Section 3 of that article holds for surfaces over any algebraically closed field).
Then the claim follows from the Push-Pull formula. For reducible correspondences, the claim extends by linearity.
\end{proof}

\begin{example} \label{ex_blowup_pfwd}
We continue Example \ref{ex_blowup}. We have $\tau_* E = 0$, so $\sigma_* \tau_* E = 0$. Yet $(\sigma \circ \tau)_* E = E$. Thus, pushforward is not functorial in general. However, we do have $\tau_* \sigma_* = (\tau \circ \sigma)_*$.
\end{example}

\subsection{Dynamical degrees}
We work in the framework of Truong's theory of dynamical degrees of correspondences \cite{MR4048444}.

Often, we have $g_* f_* \neq (g \circ f)_*$ for a composable pair of correspondences $f$ and $g$ (Example \ref{ex_blowup_pfwd}). The dynamical degree $\lambda_1(f)$ describes the ``growth'' of $(f^m)_*$ as $m \to \infty$.

 \begin{defn}[\cite{MR4048444}] \label{def_prelim_dd}
    Let $f: X \vdash X$ be a rational correspondence on an irreducible smooth projective variety $X$ of dimension $2$ over an algebraically closed field $k$. Define the \emph{degree} of $f$, relative to a fixed divisor $\Delta$, by the intersection formula
$$\deg_\Delta (f) \colonequals f_* \Delta \cdot \Delta.$$
If $\Delta$ is ample, the \emph{dynamical degree of $f$} is the limit
\begin{equation} \label{eq_dd}
    \lambda_1(f) \colonequals \lim_{m \to \infty} \deg_\Delta (f^m)^{1/m}.
\end{equation}
    The limit \eqref{eq_dd} exists and is independent of the choice of $\Delta$.
\end{defn}

\begin{defn}
    A dominant surface correspondence $f$ is \emph{algebraically stable} if, for all positive integers $m \geq 1$, we have
    $$ (f_*)^m = (f^m)_*.$$
\end{defn}

The \emph{spectral radius} of a linear endomorphism $T$ is defined as
$$\rad T \colonequals \max_{\lambda} \abs{\lambda},$$
where $\lambda$ ranges over eigenvalues of $T$.

We bound dynamical degrees using the following well-known properties. All these facts follow from standard arguments, but many appear in the literature only for special cases such as birational maps. For the sake of brevity, we give proof sketches with pointers to the literature.
\begin{thm} \label{thm_dd_props}
    \leavevmode
    \begin{enumerate}
        \item \label{it_functorial} Let $f : X \vdash Y$ and $g : Y \vdash Z$ be dominant correspondences of irreducible projective smooth surfaces. If there do not exist any contractions of $f$ that map to indeterminacy points of $g$, then
        $$g_* f_* = (g \circ f)_*.$$
        \item
        \label{it_dd_bound}
        Let $f: X \vdash X$. Then
        $$\lambda_1(f) \leq \rad{ f_* }.$$
        \item \label{it_birat_inv}
        The dynamical degree $\lambda_1(f)$ of $f: X \vdash X$ is a birational conjugacy invariant of $(X, f)$.
        \item \label{it_specialization}
        Let $S$ be a smooth integral scheme with generic point $\kappa$, and let $\pi : X \to S$ be a projective, smooth, surjective morphism of relative dimension $2$. Let $X_p$ denote the fiber at $p \in S$. Let $f: X \vdash X$ be a rational correspondence over $S$ with the following property: for each $p \in S$, the restriction $\Gamma_f|_{X_p \times X_p}$ is the graph of a dominant rational correspondence $f_p : X_p \vdash X_p$. Then, for all $p \in S$, we have
        $$ \lambda_1 ( f_p ) \leq \lambda_1 (f_\kappa).$$
        \item \label{it_not_ample}
        For any divisor $\Delta$ on $X$, for all $\epsilon$, we have as functions of $m$ that
        $$\deg_\Delta (f^m) = O((\lambda_1(f) + \epsilon)^m).$$
    \end{enumerate}
\end{thm}
\begin{proof}[Sketch]
\par(1) If $f, g$ are rational maps over $\C$, the claim is 
\cite[Proposition 1.4]{MR3342248}. Since the proof is by a computation with Chow groups (valid in any characteristic) and does not use the fact that the projections from $\Gamma_f, \Gamma_g$ are generically of degree $1$, the proof applies to correspondences in any characteristic.
\par(2) Two sequences $(a_m)$ and $(b_m)$ of real numbers are called \emph{equivalent} in this argument if there exists a constant $C > 0$ such that, for all $m$, we have $b_m/C \leq a_m \leq C b_m$. If $(a_m)$ and $(b_m)$ are equivalent, then 
$$\lim_{m \to \infty} (a_m)^{1/m} = \lim_{m \to \infty} (b_m)^{1/m}$$
provided that these limits exist.

Choose a norm $\norm{\cdot}$ on $\Num X$. The pushforward of a nef divisor class by a surface correspondence is nef. Further, for composable surface correspondences $g$ and $h$ and any nef class $\Delta$ in the domain, one has that $h_* g_* [\Delta] - (h \circ g)_* [\Delta]$ is again a nef class; this is because, on surfaces, the pullback and pushforward of a nef divisor by a proper morphism is nef, see e.g. \cite[Proof of Lemma 2.1]{MR4583005}. Therefore, if $[\Delta]$ is nef on $X$, we have that $(f_*)^m [\Delta] - (f^m)_* [\Delta]$ is nef for all $m \geq 1$, so
\begin{equation}
    \norm{(f^m)_* [\Delta]} \leq \norm{(f_*)^m [\Delta]}. \label{eq_norm_ineq}
\end{equation}

Use the notation $\norm{\cdot}$ also to denote the induced operator norm. A Perron-Frobenius argument shows, for any ample class $[\Delta]$, that
$ \norm{(f_*)^m [\Delta]} $ and $\norm{(f_*)^m}$
are equivalent. We also have that the sequences $ \norm{(f^m)_* [\Delta]} $ and $\deg_{\Delta} f^m$ are equivalent (follows from \cite[p.151]{MR4048444}).

Putting these facts together, if $[\Delta]$ is an ample divisor class, we have
\begin{align*}
    \lambda_1 (f) &= \lim_{m \to \infty} (\deg_{\Delta} f^m)^{1/m} \\
                  &= \lim_{m \to \infty} \norm{(f^m)_* [\Delta]}^{1/m} \\
                  &\leq \lim_{m \to \infty} \norm{(f_*)^m [\Delta]}^{1/m} & \text{(by \eqref{eq_norm_ineq})} \\
                  &= \lim_{m \to \infty} \norm{(f_*)^m }^{1/m} & \text{(by equivalence)} \\
                  & = \rad f_* & \text{(a consequence of Jordan normal form)}.
\end{align*}
\par(3) \cite[Theorem 1.1]{MR4048444}.
\par(4) Given a line bundle $L$ on $X$ that is nef over $S$, let $L_p = L|_{X_p}$. Noting Definition \ref{def_prelim_dd}, it suffices to prove that for any such $L$ and any $p \in S$, we have
$$\deg_{L_p} (f_p) \leq \deg_{L_\kappa} (f_\kappa).$$
When $X \to S$ is smooth and $f$ is birational (a rational 1-1 correspondence) over $S$, this is \cite[Lemma 4.1]{MR3332894}, with the main ingredient of the proof being the constancy of intersection multiplicities in flat families. The proof given there generalizes straightforwardly to arbitrary correspondences.
\par(5) If $\Delta$ is ample, then 
$$\lim_{m \to \infty} (\deg_\Delta f^m)^{1/m} = \lambda_1(f).$$
If $\Delta$ is any divisor, then it can be expressed as the difference of ample divisors, so
$$\lim_{m \to \infty} (\deg_\Delta f^m)^{1/m} \leq \lambda_1(f).$$
Then apply Definition \ref{def_prelim_dd}.
\end{proof}
\section{Secant, Reflection, and Billiards} \label{sect_basic}

This section provides the ``rules of play'' for algebraic billiards. These results are used throughout the paper. 

We work on the domain $C \times D$, where $C$ is the table and $D$ is a standard unit conic. We give formal definitions of the secant and reflection correspondences on $C \times D$, denoted $s$ and $r$. The billiards correspondence (Definition \ref{def_b}) is the composition 
$$b \colonequals r \circ s.$$
We describe the contracted and exceptional curves for $s$, $r$, and $b$ and describe their basic geometry.

For maximum flexibility, we set up our definitions to accommodate singular and reducible tables. the reader will get the main ideas by considering just smooth, irreducible tables.

\subsection{The idea}

Let us briefly explain how we set up algebraic billiards. Note that most works on complex billiards use the domain $C \times \PP^1$, e.g. \cite{MR3224419}. Our setup is a little different. In particular, in our setup, if the base field is $\C$, a real ray and its opposite ray are distinct elements of $D$.

First, we describe real billiards so that it sounds as algebraic as possible.

Let $C(\R) \subset \R^2$ be the real part of a complex algebraic plane curve $C$. Suppose that $C(\R)$ is smooth and bounded, but not necessarily convex.

A real direction is defined as a point $(q_0, q_1) \in \R^2$ such that $q_0^2 + q_1^2 = 1$; this will play the role of the direction of motion of the particle. If $D$ is the complex conic with affine equation $q_0^2 + q_1^2 = 1$, then $D(\R)$ is the unit circle in $\R^2$, the space of real directions.

Billiards is a composition of two operations.
\begin{enumerate}
    \item Secant: let $(c, q) = (x_0, x_1, q_0, q_1)$ be a ``general'' point of $C(\R) \times D(\R)$; it contains the data of position and direction. Assume $q$ is facing the interior of the table. The ball moves to the first intersection $(x'_0, x'_1)$ of $C(\R)$ with the ray at $c$ of direction $q$. The direction does not change. 
    
    We can write down an algebraic equation this intersection point satisfies: the point-slope form of the equation of the line through $(x_0, x_1)$ with slope $q_1 / q_0$ is 
\begin{equation}
    x'_1 - x_1 = \frac{q_1}{q_0} (x'_0 - x_0). \label{eq_point_slope_motivation}
\end{equation}
    \item Reflection: this operation fixes the point $c = (x_0, x_1) \in C(\R)$, and replaces $(q_0, q_1)$ by its reflection $(q'_0, q'_1)$ across the tangent line to $C$ at $x$. This is often stated as ``the angle of incidence equals the angle of reflection,'' but we can frame this relation algebraically. First notice that we do not care about the actual tangent line at $c$, just its slope, denoted $t_c$. Then we have the convenient formula
    \begin{equation}
    q'_1 - q_1 = -\frac{1}{t_c} (q'_0 - q_0). \label{eq_point_slope_r_motivation}
\end{equation}
The slope $-1/t_c$ appearing in \eqref{eq_point_slope_r_motivation} is the slope of the normal to $C$ at $x$.
Indeed, over $\R$, the point $(q'_0, q'_1)$ is the intersection on $S^1$ of the line through $(q_0, q_1)$ of the slope of the normal to $C$ at $x$; see Figure \ref{fig_r}.
\end{enumerate}
Now we note what changes over an algebraically closed field $k$.
\begin{itemize}
\item We let $C$ be a plane curve over $k$, and $D$ be the conic $q_0^2 + q_1^2 = 1$ defined over $k$.
\item We use \eqref{eq_point_slope_motivation} and \eqref{eq_point_slope_r_motivation} to define graphs of rational correspondences, called $s$ and $r$, on an affine subset of $C \times D$. These extend secant and reflection to almost the whole domain. 
\item To have a projective domain, we homogenize \eqref{eq_point_slope_motivation} and \eqref{eq_point_slope_r_motivation}. This lets us extend secant and reflection at the ``points at infinity'' of $C$ and $D$.
\item As described so far, secant is $d$-valued and reflection is $2$-valued (note that \eqref{eq_point_slope_motivation} and \eqref{eq_point_slope_r_motivation} have trivial solutions $c = c'$ and $q = q'$, respectively). We remove these trivial solutions, so for us, secant is $(d-1)$-valued and reflection is $1$-valued. See Figure \ref{fig_s} and Figure \ref{fig_r}.
\item There are indeterminacy points for the secant and reflection operations related to points of $C$ and $D$ at infinity.
\item Finally, billiards is defined by $b \colonequals r \circ s$.
\end{itemize}

A formal account follows in the next section.

\begin{figure}
\includegraphics[width=3in]{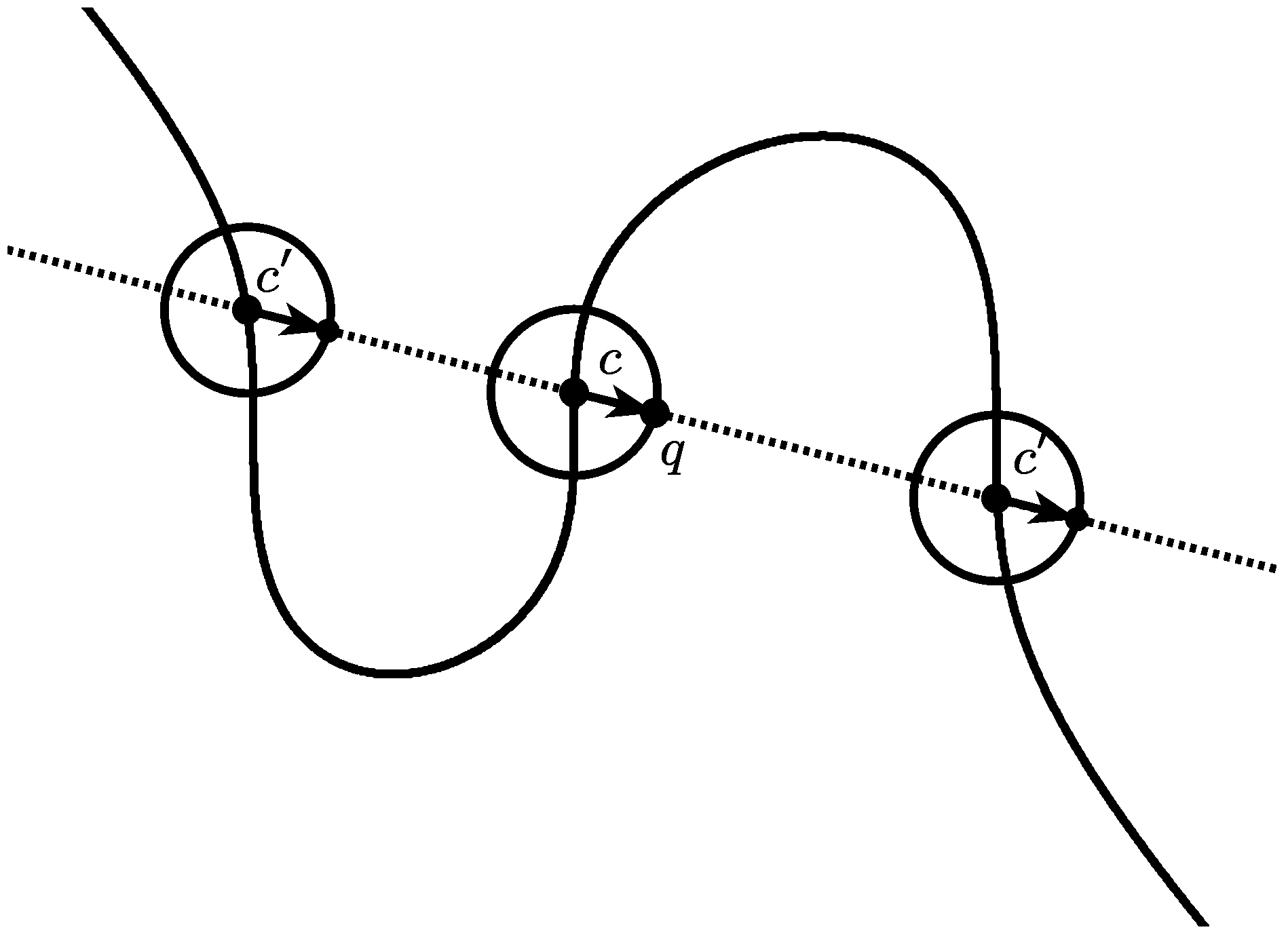}
\caption{The secant correspondence $s$ on the real locus of a plane cubic $C$. A point in the domain is represented as a unit-length vector based on $C$, given by a base point $c \in C$ and direction $q \in S^1$. The two images of $(c, q)$ are the points $(c', q)$. Copies of $S^1$ are shown based at each $c'$ to show that the value of $q$ does not change upon applying $s$.}
\label{fig_s}
\end{figure}

\begin{figure}
\includegraphics[width=2.5in]{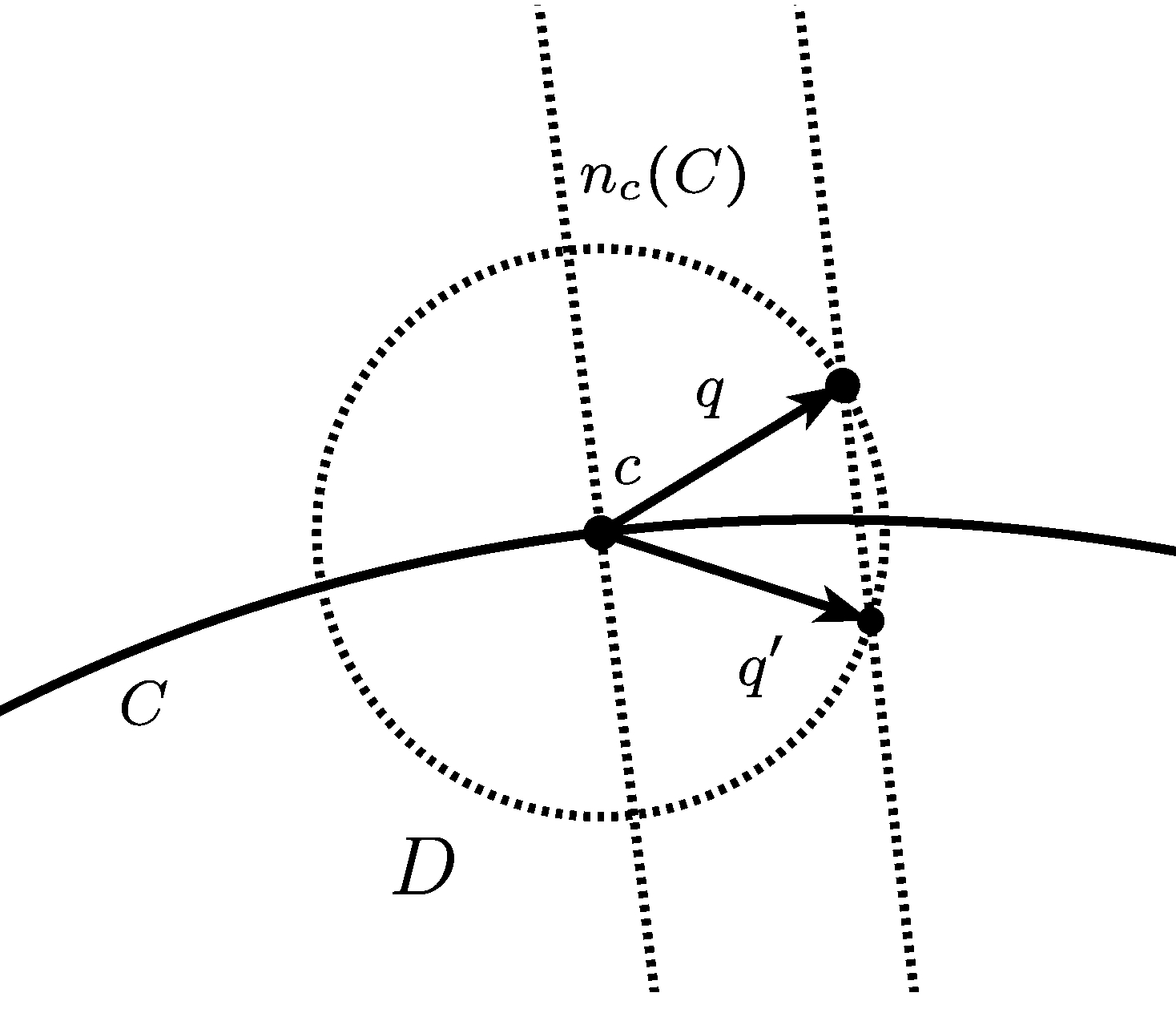}
\caption{The reflection correspondence $r$ on the real locus of a plane curve $C$. The point $(c, q)$ is mapped to $(c, q')$. The point $q'$ is the point on $D$ beside $q$ on the line through $q$ parallel to the normal $n_c(C)$ through $c$ at $C$.}
\label{fig_r}
\end{figure}

\subsection{Secant and reflection}
Let $k$ be an algebraically closed field with $\charac k \neq 2$. Let $x_0, x_1$ be coordinates on $\A^2$. We embed $\A^2$ in $X \colonequals \PP^2$, with homogeneous coordinates $[X_0 : X_1 : X_2]$; so
$$x_0 = X_0 / X_2, \quad x_1 = X_1 / X_2.$$
Let $X_\infty$ be the line at infinity in $X$, i.e. $X_\infty = X \smallsetminus \A^2$.

The tangent bundle $T \A^2$, viewed as a variety, is isomorphic to $\A^4$; further, there is a preferred choice of coordinates
$$ T \A^2 \cong \A^4 = \Spec k[x_0, x_1, q_0, q_1].$$
Formally, $q_0 = dx_0$, $q_1 = dx_1$ are the Zariski cotangent vectors induced by $x_0$ and $x_1$. We treat them as formal generators of a symmetric algebra in order to give the tangent bundle the structure of a variety.

Let $Q \colonequals \PP^2$, and let $[Q_0 : Q_1 : Q_2]$ be homogeneous coordinates on $Q$. We compactify $T \A^2$ to 
$$X \times Q = \PP^2 \times \PP^2.$$
Let $Q_\infty$ be the line at infinity in $Q$, i.e. $Q_\infty = Q \smallsetminus \A^2$.

The \emph{slope} of a line $\ell \neq X_\infty$ in $X$ is the unique point in $X_\infty$ that lies on $\ell$. A nonzero tangent vector $q$ at a point $x \in \A^2$ determines an affine line in $X$, the slope of which depends only on the scaling class of $q$ and not on $x$. Going in the other direction, the pencil of affine lines in $X$ of a given slope all have the same projectivized tangent vector at every point. These operations are inverses, so there is a natural identification 
$$X_\infty \xrightarrow{\raisebox{-0.7ex}[0ex][0ex]{$\sim$}} Q_\infty.$$
This identifies $[X_0 : X_1 : 0]$ with $[Q_0 : Q_1 :0]$ if and only if $[X_0~:~X_1]~= [Q_0 : Q_1]$. 

We therefore define the \emph{space of slopes} to be $\PP^1$, which is identified both with $X_\infty$ and $Q_\infty$. It will occasionally be helpful to write the slope $[1:m]$ simply as $m$. One may check that for any $m,b \in k$, the slope of the line defined by $x_1 = mx_0 + b$ is $m$, as expected.

Let $C_0 \subset X$ be a projective plane curve (i.e. a sum of irreducible curves, possibly with multiplicity, not necessarily smooth, of dimension 1 over $k$). Let $C \to C_0$ be a desingularization. Then $C$ is a disjoint union of irreducible smooth curves with multiplicity, equipped with a map $C \to X$.

Let $C_\infty$ be the set of points of $C$ lying above $X_\infty$. We call these the \emph{points of $C$ at infinity}.

\begin{defn}
Assume that $C_0$ does not contain the line at infinity. The \emph{tangent slope map}
$$ t: C \to \PP^1, $$
$$t(c) = [t(c)_0 : t(c)_1]$$
is defined as follows. The tangent slope $t(c)$ of any nonsingular point $c \in C_0 \cap \A^2$ is the unique intersection point of the tangent line (i.e. embedded tangent space) to $C_0$ at $c$ with $X_\infty$. This gives us a rational map from $C$ to $X_\infty \simeq \PP^1$, which extends uniquely to a morphism. (Intuitively, above singular points, one may calculate $t$ by a limit.)
\end{defn}

So far all our definitions have only made use of the intrinsic structure of lines in $\PP^2$ with a marked line at infinity. To define reflection, we need to make a choice of quadratic form. We work with the quadratic form $q_0^2 + q_1^2$.

The \emph{space of unit tangent vectors} is the affine quadric cut out by the equation by $q_0^2 + q_1^2 = 1$. We compactify it to the projective plane quadric $D \subset Q$ defined by
$$D : \; Q_0^2 + Q_1^2 = Q_2^2.$$
Since $\charac k \neq 2$, the quadric $D$ is irreducible and reduced. 

The points of $D$ shall be called \emph{directions}, since they ultimately play the role of directions in billiards.

Let $D_\infty = D \cap Q_\infty$. We call these points the \emph{isotropic directions.} Let $i$ be a square root of $-1$ in $k$; then,
$$D_\infty = \{[1 : i : 0], [1 : -i : 0]\}.$$
Viewed as slopes, these points are $\pm i$.

\begin{defn}
Recall that $\charac k \neq 2$. The \emph{slope} of a direction $q \in D$ is its image $[q]$ by the 2-to-1 ramified cover
\begin{align*}
[ \cdot ] : & \; D \to \PP^1, \\
&[Q_0 : Q_1 : Q_2] \mapsto [Q_0 : Q_1].
\end{align*}
\end{defn}

Under the identification $Q_\infty \simeq \PP^1$, the slope map sends a point in $D \smallsetminus D_{\infty}$ to its projective class as a tangent vector. Notice that the isotropic directions are the ramification points of the slope map. The isotropic directions are so named because any tangent vector in $T \A^2$ of slope $\pm i$ is isotropic (self-normal) for the quadratic form $q_0^2 + q_1^2$.

\begin{defn}
The \emph{directed phase space}, or henceforth just \emph{phase space}, is $C \times D$.
\end{defn}

We construct relations on $X \times Q$ as subvarieties of $(X \times Q)^2$. Coordinates on the second factor $X \times Q$ will be written
$$[X'_0 : X'_1 : X'_2], \quad [Q'_0 : Q'_1 : Q'_2].$$

\begin{defn} \label{def_secant_relation}
The \emph{secant relation} $S$ on $X \times Q$
is the Zariski closure in $(X \times Q)^2$ of the locus $\{(x, q, x', q')\}$ determined by
\begin{align*}
&x \neq x', \\
&q = q', \\
&\begin{vmatrix}
X_0 & X'_0 & Q_0 \\
X_1 & X'_1 & Q_1 \\
X_2 & X'_2 & 0 \\
\end{vmatrix} = 0.
\end{align*}
\end{defn}

The secant relation is not a correspondence, because it is not generically finite; it has 1-dimensional and 2-dimensional fibers above points in both the first and second factor. Indeed, the fiber of $S$ above a general choice of $[X_0 : X_1 : X_2]$ and $[Q_0 : Q_1 : Q_2]$ is parametrized by the line through $[X_0 : X_1 : X_2]$ and $[Q_0 : Q_1 : 0]$, which is the line in $\PP^2$ through $[X_0 : X_1 : X_2]$ of slope $[q]$. The 2-dimensional fibers occur (1) when $[Q_0 : Q_1 : Q_2] = [0:0:1]$, and (2) when $X_2 = 0$ and $[X_0 : X_1] = [Q_0 : Q_1]$. In both exceptional cases, the fiber is dimension 2. Thus $S$ is a $5$-dimensional subvariety of the $8$-dimensional ambient space $(X \times Q)^2$.

\begin{defn} \label{def_secant}
    Assume that $C_0$ is irreducible of degree $d \geq 2$. The \emph{secant correspondence} is a rational correspondence $s : C \times D \vdash C \times D$. Its graph $\Gamma_s \subset (C \times D)^2$ is
    defined as follows. Let $S_{C\times Q}$ be the pullback to $(C \times Q)^2$ of the secant relation $S \subset (X \times Q)^2$. Then let 
    \begin{equation} \label{eq_s_def}
        \Gamma_s \colonequals S_{C \times Q}|_{(C \times D)^2}.
    \end{equation}
\end{defn}

\begin{remark}
The following definition is easily seen to be equivalent to Definition \ref{def_secant}. Let $s_+ : C \times D \vdash C \times D$ be the correspondence with graph $\Gamma_{s_+}$ defined by $q = q'$ and $\det(x, x', [Q_0 : Q_1 : 0]) = 0$, as in Definition \ref{def_secant_relation}; this graph contains the diagonal $\Gamma_1$. Then let
$$\Gamma_s = \Gamma_{s_+} - \Gamma_1.$$
\end{remark}

In this paper, we mostly work with dominant correspondences on irreducible varieties. Remark \ref{rem_reduc} explains how to extend the notion of dominant correspondence to reducible domains. Since many billiard tables of interest in the literature are reducible as varieties, i.e. stadia, we here consider such domains.

\begin{defn} \label{def_s_reducible}
    We can extend the definition of the secant correspondence to reducible curves $C_0$ of degree $d \geq 2$, as long as $C_0$ contains no multiple lines, as follows.
    
    If $C_0$ contains no lines, then \eqref{eq_s_def} suffices to define $s$. If $C_0$ contains only lines of multiplicity $1$, then we write $C$ as a union of a line-free component $\tilde{C}$ of degree $d \geq 2$, together with the lines $\ell_1, \ldots, \ell_\nu$ for some $\nu$. Define
    $$\Gamma_{\tilde{s}} \colonequals S_{C \times Q}|_{(\tilde{C} \times D)^2}.$$
    For each $i$, define
    $$\Gamma_{s_i} \colonequals S_{C \times Q}|_{(\tilde{C} \times D) \times (\ell_i \times D)}.$$
    For each distinct pair $(i, j)$, define
    $$\Gamma_{s_{(i, j)}} \colonequals S_{C \times Q}|_{(\ell_i \times D) \times (\ell_j \times D)}.$$
    Then we define
    $$\Gamma_s \colonequals \Gamma_{\tilde{s}} + \sum_i \Gamma_{s_i} + \sum_{i, j: \; i \neq j} \Gamma_{s_{i,j}}.$$
\end{defn}
Note that we cannot naively define the secant correspondence with \eqref{eq_s_def} when $C_0$ contains a line. Indeed, if $C_0$ contains a line $\ell$ of slope $[q]$, then ${S_{C \times Q}|}_{{(C \times D)}^2}$ contains a $2$-dimensional irreducible component $(\ell \times q)^2$ with non-dominant projection to each factor $C \times D$, contravening the definition of dominance for correspondences (Definition \ref{def_rat_corr}). 

We collect some basic properties of the secant correspondence in the following proposition.

\begin{prop}[properties of secant] \label{prop_basic_s} \leavevmode
Let $C_0 \subseteq \PP^2$ be an algebraic curve of degree $d \geq 2$. Let $C \to C_0$ be a resolution of singularities. 
\begin{enumerate}
    \item The secant correspondence is a self-adjoint, $(d-1)$-to-$(d-1)$, dominant rational correspondence 
    $$s: C \times D \vdash C \times D.$$
    For any $(c, q) \in C \times D$ outside $\Ind s$, the image $s(c,q)$ is the multiset of $d-1$ points
    $$s(c,q) = (C \cap \ell \smallsetminus \{c\}) \times q,$$
    where $\ell$ is the line through $c$ of slope $[q]$.
    \item The projection $C \times D \to D$ is $s$-invariant.
    \item \label{prop_basic_s_contracteds} If $C_0$ is irreducible, the contractions of $s$ are of the following form. For each pair $c_\infty \in C_\infty$ and $q \in D$ such that $[q] = t(c_\infty)$, there exists a contraction
        $$(C \times q) \times (c_\infty, q) \subset \Gamma_s.$$    
    \item If $C_0$ is irreducible, the expansions of $s$ are of the following form. For each pair $c_\infty \in C_\infty$ and $q \in D$ such that $[q] = t(c_\infty)$, there exists an expansion
    $$(c_\infty, q) \times (C \times q) \subset \Gamma_s.$$    
    Thus the indeterminacy locus $\Ind_s \subset C \times D$ is the set
    $$\{ (c_\infty, q ) : c_\infty \in C_\infty, \; [q] = t(c)\}.$$
\end{enumerate}
\end{prop}

\begin{proof}
\par(1) We give the argument if $C$ is reduced and irreducible. The general case is similar.

Counting the defining equations of $\Gamma_s$, we see that each irreducible component of $\Gamma_s$ has dimension at least $2$. Since $q = q'$ on $\Gamma_s$, the fibers of the first projection $\pi_1: \Gamma_s \to C \times D$ are all $0$-dimensional or $1$-dimensional. The fibers are generically $0$-dimensional. To see this, let $(c, q)$ be a general point of $C \times D$; in particular, assume that $c$ is not in $C_\infty$ and that $[q] \neq t(c)$. The fiber $\pi_1^{-1}(c, q) \subset X \times Q$ is parametrized by the line $\ell$ of slope $[q]$ through $c$. Since $C$ does not contain $\ell$ by assumption, the line $\ell$ has $d - 1$ intersection points (with multiplicity) with $C$ besides $c$ itself, by B\'ezout's theorem, so $\pi_1^{-1}(c, q) \cap C \times D$ contains $d - 1$ points with multiplicity. Since the generic fiber is $0$-dimensional and the maximal fiber dimension is $1$, we have $\dim \Gamma_s = 2$.

Now, if $\Gamma_s$ were not dominant, there would be an irreducible component $\Gamma$ of $\Gamma_s$ with non-dominant projection to both factors. This could only occur if there were a $1$-dimensional set in $C \times D$ with $1$-dimensional $\pi_1$-fibers. But since $C$ contains no lines, the set with $1$-dimensional $\pi_1$-fibers consists only of $(c, q)$ such that $c \in C_\infty$ and $[q] = t(c)$. Thus $\Gamma_s$ is dominant, and since each component has dominant projections, the projections are generically finite by a dimension count. This shows that $\Gamma_s$ is a dominant rational correspondence. Since $S$ is a symmetric relation, so is $s$, proving self-adjointness. 

Finally, the fact that $s$ is $(d-1)$-to-$(d-1)$ follows from B\'ezout's theorem, as above.
\par(2) Immediate.
\par(3) The existence of such curves in $\Gamma_s$ is immediate from Definition \ref{def_secant}; they are contractions since $C_0$ is irreducible. To see that these are all the contractions, let $(c, q) \in C \times D$. If $(c', q') \in C \times D$ is not of the form $c' = c_\infty$ and $[q'] = t(c')$, then the fiber of $\Gamma_S$ above $\pi_{2}^{-1}(c',q')$ is parametrized by the unique line of slope $[q']$ through $c'$. By assumption $C_0$ contains no lines, so there are finitely many intersection points by B\'ezout's theorem.
\par(4) Immediate from (1) and (3), since expansions are contractions of the adjoint correspondence.
\end{proof}

Now we define a second correspondence $r$ on $C \times D$ that carries out reflection. The construction is closely related to that of $s$, but with some minor differences.

\begin{defn}
Assume that $C$ and does not contain the line at infinity. Then each point $c \in C$ has a \emph{normal slope}
$$n(c) = [n(c)_0 : n(c)_1] \in \PP^1,$$
defined by
$$[n(c)_0 : n(c)_1] \colonequals [-t(c)_1 : t(c)_0].$$
It is so named because any vector $(v_1, v_2) \in k^2$ pairs to $0$ with $(-v_2, v_1)$ via the bilinear pairing associated to the quadratic form $q_0^2 + q_1^2$ defining $D$.

An \emph{isotropic tangency point} of $C$ is a point $c \in C$ such that $t(c)$ is an isotropic slope $\pm i$. These are exactly the points for which $t(c) = n(c)$.
\end{defn}

\begin{defn} \label{def_reflection_relation}
Assume that $C$ is separable and does not contain the line at infinity. The \emph{reflection relation}
$$R_{C \times Q} : C \times Q \vdash C \times Q$$
is the Zariski closure of the locus of $(c, q, c', q') \in (C \times Q)^2$ determined by
\begin{align*}
&c = c',\\
&q \neq q',\\
&\begin{vmatrix}
Q_0 & Q'_0 & n(c)_0 \\
Q_1 & Q'_1 & n(c)_1 \\
Q_2 & Q'_2 & 0 \\
\end{vmatrix} = 0.
\end{align*}
\end{defn}
Notice that, unlike $S$, there is no reasonable definition of ``reflection'' on $X \times Q$ without reference to the curve $C$, so we start with $R_{C \times Q}$ instead.

In parallel with Definition \ref{def_secant_relation}, the reflection relation is not a correspondence, since it has $1$-dimensional and $2$-dimensional fibers above points in both the first and second factor.
Over a general point $(c, q)$, the fiber of $R$ is parametrized by a line in $Q$, a $1$-dimensional fiber.
The $2$-dimensional fibers occur when $[Q_0 : Q_1 : Q_2] = [n(c)_0 : n(c)_1 : 0]$, in which case the fiber is 2-dimensional. Thus $R_C$ is a 4-dimensional subvariety of the 6-dimensional space $(C \times Q)^2$.

\begin{defn} \label{def_reflection}
    Recall that $\charac k \neq 2$. Let $C_0$ be an algebraic plane curve of degree $d \geq 1$ that does not contain the line at infinity or any affine line with an isotropic slope.
    
    The \emph{reflection correspondence} is a rational correspondence $r: C \times D \vdash C \times D$. Its graph $\Gamma_r \subset (C \times D)^2$ is defined as
    $$\Gamma_r \colonequals R_{C \times Q}|_{(C \times D)^2}.$$
\end{defn}

\begin{remark}
The following definition is easily seen to be equivalent to Definition \ref{def_reflection}. Let $r_+ : C \times D \vdash C \times D$ be the correspondence with graph $\Gamma_{r_+}$ defined by $x = x'$ and $\det(q, q', [n(c)_0 : n(c)_1 : 0]) = 0$, as in Definition \ref{def_reflection_relation}; this graph contains the diagonal $\Gamma_1$. Then let
$$\Gamma_r = \Gamma_{r_+} - \Gamma_1.$$
\end{remark}

We collect basic properties of the reflection correspondence $r$ in the following proposition. Notice the near-duality with the properties of $s$ in Proposition \ref{prop_basic_s}.
\begin{prop}[properties of reflection] \label{prop_basic_r} Let $C_0$ be an algebraic plane curve as in Definition \ref{def_reflection}, with desingularization $C$.
\begin{enumerate}
\item \label{prop_basic_r_bir} The reflection correspondence $r$ is a birational involution, i.e. a self-adjoint, $1$-to-$1$, dominant rational correspondence
$$r : C \times D \dashrightarrow C \times D.$$
For any $(c, q) \in C \times D$ outside $\Ind r$, we have
$$r(c,q) = (c, q'),$$
where, writing $\ell$ for the line in $Q$ through $q$ of slope $[n(c)]$, we have 
$$q' = D \cap \ell \smallsetminus \{q\} \quad \text{(with multiplicity)}.$$
\item \label{prop_basic_r_invar} The projection $C \times D \to C$ is $r$-invariant.
\item \label{prop_basic_r_ind} 
The contractions of $r$ are of the following form. For each isotropic tangency point $c \in C$ of tangent slope $q_\infty \in D_\infty$, there exists a contraction
$$(c \times D) \times (c, q_\infty) \subset \Gamma_r.$$
\item \label{prop_basic_r_contracted} 
The expansions of $r$ are of the following form. For each isotropic tangency point $c \in C$ of tangent direction $q_\infty \in D_\infty$, there exists an expansion
$$(c, q_\infty) \times (c, D) \subset \Gamma_r.$$
\item
The birational map $r$ exchanges the isotropic directions, in the sense that for each $I \in D_\infty$, we have a curve
$$\{(c, I, c, -I) : c \in C\} \subset \Gamma_r.$$
\label{prop_basic_r_isot}
\end{enumerate}
\end{prop}
\begin{proof}
We follow the proof of Proposition \ref{prop_basic_s}.

\par(1) We claim that $r$ restricts to a self-adjoint, dominant $1$-to-$1$ correspondence on each $C_{(j)} \times D$, where $C_{(j)}$ is an irreducible component of $C$. Then it follows that $r$ is a self-adjoint, dominant $1$-to-$1$ correspondence on $C \times D$. The only difference is in showing that there is no irreducible component $\Gamma$ of $\Gamma_r$ with non-dominant projection to both factors. Let $\pi_1: \Gamma_r \to C_{(j)} \times D$ be the first projection. The subset of $C_{(j)}$ with $1$-dimensional $\pi_1$-fibers consists only of $(c,q)$ such that $q \in D_\infty$ and $[q] = t(c)$ or $c \in C_\infty$. Since there are no isotropic lines in $C$ and the line at infinity is not in $C$, the set of all such $(c, q)$ is $0$-dimensional.

\par(2) Immediate.

\par(3) The existence of such curves in $\Gamma_r$ is immediate from Definition \ref{def_reflection}; they are contractions since $D$ is irreducible. To see that these are all the contractions, let $(c,q) \in C \times D$. If $(c', q') \in C \times D$ is not of the form $q' \in D_\infty$ and $[q'] = t(c) = n(c)$, then the fiber of $\Gamma_R$ above $\pi_2^{-1}(c',q')$ is parametrized by the unique line in $Q$ of slope $[n(c)]$ through $q$. Since $D$ is an irreducible conic, there are finitely many intersections.

\par(4) Immediate from (1) and (3), since expansions are contractions of the adjoint correspondence.

\par(5) Immediate from the definitions of $R$ and $r$.
\end{proof}

\begin{remark}
    The secant and reflection correspondences have indeterminacy points that sit inside contracted curves; in Birkett's language, they are \emph{tangled} \cite{birkett2022stabilisation}.
\end{remark}

\subsection{The billiards correspondence}

\begin{defn} \label{def_b}
An \emph{algebraic billiard curve} is a plane curve $C_0 \subset \PP^2 = X$ of degree $d \geq 2$, possibly singular and reducible, not containing any multiple lines, any isotropic line, or the line at infinity. Let $C$ be a resolution of singularities of $C_0$. Let $s$ be the secant correspondence and let $r$ be the reflection correspondence, as just defined. The \emph{billiards correspondence} $b$ is the dominant rational correspondence
\begin{align*}
b &: C \times D \vdash C \times D,\\
b &= r \circ s.
\end{align*}
Note that the possibility of composing $r$ and $s$ is justified by the fact that these are dominant correspondences (Proposition \ref{prop_basic_s}, Proposition \ref{prop_basic_r}).
\end{defn}

The basic properties of algebraic billiards, and its relationship to billiards in other contexts, are summarized in the following proposition.
\begin{prop} \label{prop_b_basic}
Let $C_0 \subset \PP^2$ be an algebraic billiard curve, and let $C \to C_0$ be a desingularization. Let $b : C \times D \vdash C \times D$ be the billiards correspondence.
\begin{enumerate}
    \item The correspondence $b$ is $(d-1)$-to-$(d-1)$.
    \item The correspondence $b$ is reversible, in the sense that it is birationally conjugate via $r$ to its own adjoint.
    \item There is a generically 2-to-1 semiconjugacy from $b$ to a correspondence 
    $$\bar{b} : C \times \PP^1 \vdash C \times \PP^1.$$
    When $k = \C$, the correspondence $\bar{b}$ agrees with complex billiards \cite[Remark 1.6]{MR3236494}.
    \item \label{it_real_inside_complex} Suppose $k = \C$ and $T$ is a finite union of segments $T_1, \ldots, T_n$ of real curves in $\R^2$. Let $b_T$ be the classical billiards map (partially defined) on $T \times S^1$. Suppose that the Zariski closure $\bar{T}$ of $T$ in $\PP^2$ is an algebraic curve. Then $\bar{T}$ is an algebraic billiard curve. Taking $C_0 = \bar{T}$, then 
    $$\Gamma_{b_T} \subset \Gamma_b.$$
    If $C_0$ is smooth and irreducible, the graph $\Gamma_b$ is the Zariski closure of $\Gamma_{b_T}$ in $C \times D$. Thus Definition \ref{def_intro_b} is a special case of Definition \ref{def_b}. 
\end{enumerate}
\end{prop}
\begin{proof}
    We use Proposition \ref{prop_basic_s} and Proposition \ref{prop_basic_r}.
    
    \par(1)\enspace Since $s$ is $(d-1)$-to-$(d-1)$ and $r$ is $1$-to-$1$, their composite is $(d-1)$-to-$(d-1)$.
    \par(2)\enspace  Let the adjoints of $b, s, r$ be denoted $b^{\adj}, s^{\adj}, r^{\adj}$. We have
    \begin{align*}
        r^{-1} \circ b \circ r &= r^{-1} \circ r \circ s \circ r & \text{(definition of $b$)} \\
        &= s \circ r & \text{(since $r$ is a birational map)} \\
        &= s^{\adj} \circ r^{\adj} & \text{(self-adjointness of $s$ and $r$)} \\
        &= b^{\adj} & \text{(definition of $b$)}.
    \end{align*}

    \par(3)\enspace  Let 
    $$\mu : C \times D \to C \times D,$$
    $$ (c, [Q_0 : Q_1 : Q_2]) \mapsto (c, [-Q_0 : -Q_1 : Q_2]).$$
    Then $b$ commutes with $\mu$. The quotient of $D$ by the involution $[Q_0 : Q_1 : Q_2] \mapsto [-Q_0 : -Q_1 : Q_2]$ is the map $D \to \PP^1$ defined by $q \mapsto [q]$, so the quotient of $C \times D$ by $\mu$ is isomorphic to $C \times \PP^1$. Since $b$ commutes with $\mu$, it descends to a correspondence $\bar{b}$.

    We check the second claim on a general subset of the domains. The definition of complex billiards as a correspondence in \cite[Remark 1.6]{MR3236494} is as follows. Given a line $\ell$ of non-isotropic slope in $\C^2$, the reflection across that line is the unique complex isometry (for $dx_0^2 + dx_1^2$) fixing $L$. As an action on the space $\PP^1$ of slopes, reflecting in a line $L$ of non-isotropic slope $[q]$ is equivalent to applying the unique involution on $\PP^1$ that fixes $[q]$ and $[n(q)]$. Given a general point $c \in C$ and $[q] \in \PP^1$, the images of $(c, [q])$ are those points $(c', [q'])$ such that $c' \in C$ is on the line $\ell$ of slope $[q]$ through $c$, and $[q']$ is the complex reflection of $[q]$ across $L = T_{c'} C$. Since $\bar{b}$ also fixes $[q]$ and $[n(q)]$ on $\PP^1$, the two definitions agree.
    
    \par(4)\enspace  It is clear that $\bar{T}$ contains no multiple lines since it is a set-theoretic Zariski closure of real curves. Since $T$ is a subset of a curve in $\R^2$, we know that $\bar{T}$ contains no lines at infinity or isotropic lines. So $\bar{T}$ is an algebraic billiard table and Definition \ref{def_b} is applicable.
    
    Consider any pair $(c, q) \in T \times S^1$ such that $b_T(c, q)$ is defined. The point $(c', q') \colonequals b_T(c,q)$ has the property $c \neq c'$, and clearly $(c, q, c', q) \in \Gamma_s$, so $(c, q, c', q') \in \Gamma_b$. Thus $b_T(c,q)$ is in the multiset $b(c,q)$, so $\Gamma_{b_T} \subseteq \Gamma_b$.

    Now suppose further that $C_0$ is smooth and irreducible, and of degree $d$. We claim that $\Gamma_b$ is irreducible. Since $r$ is birational, it suffices to show that $\Gamma_s$ is irreducible.
    
    To see this, consider the projection $\psi: \Gamma_s \to D$. We claim that over a Zariski open subset $U$ of $D$, the fibers of $\psi$ are irreducible curves. A map of varieties with irreducible codomain and equidimensional irreducible fibers must have an irreducible domain, so this shows that the subset of $\Gamma_s$ above $U$ is irreducible. Since $\Gamma_s$ is a dominant correspondence, it is equal to the Zariski closure of its restriction over $U$.
    
    Let us check that $\psi$ has irreducible fibers. Given any $q = [Q_0 : Q_1 : Q_2] \in D$, let $\phi_q: C \to \PP^1$ be the map induced on $C$ by the projection $\PP^2 \to \PP^1$ with center $[Q_0 : Q_1 : 0]$.
    Let $U$ be the set of $q \in D$ such that $\phi_q$ is of degree $d$ and all ramification points are simple and branched over distinct points in $\PP^1$. Then $U$ is a Zariski open set of $D$. 
    
    The curve $\psi^{-1}(q)$ is irreducible if and only if the monodromy of the cover $\phi_q$ is 2-transitive. This monodromy is the symmetric group on $d$ letters because the branching of the cover is simple and distict. Since $d > 1$, the monodromy is 2-transitive.

    Now, observe that $\Gamma_{b_T}$ is a 2-real-dimensional subset of the irreducible variety $\Gamma_b$. If its Zariski closure were not all of $\Gamma_b$, then it would be contained in a complex curve $W \subset \Gamma_b$. Then the image of $W$ in $C \times D$ would be an algebraic curve containing $T \times S^1$. But the Zariski closure of $T \times S^1$ is complex 2-dimensional.
\end{proof}

\begin{remark}
    The reflection map is defined with reference to the standard quadratic form $q_0^2 + q_1^2$. Because we are working over an algebraically closed field of characteristic not $2$, any other choice of quadratic form can be transformed to $q_0^2 + q_1^2$ by a change of variables over $k$. (However, such coordinate changes act simultaneously on $C$ and $D$.) By choosing $q_0^2 + q_1^2$, we extend the standard Euclidean billiard on the real locus, whereas other choices extend pseudo-Euclidean billiards \cite{MR2518642}.
\end{remark}

\begin{remark} \label{rem_near_duality}
    The definitions of $r$ and $s$ are in fact completely symmetric. (If we consider the standard metric on the tangent space, and use it to define a normal, then $q = n(q)$ on $D$.) Some of our arguments ought to have interesting generalizations to billiards-like correspondences on appropriate pairs of curves $C$ and $D$, in which case we expect there would be some similarities to Finsler billiards \cite{MR1866992}. Since our motivation is classical billiards, in this article we only consider the standard conic $D$.
\end{remark}

\section{The invariant form} \label{sect_invariant_form}
Real billiards preserves a nonvanishing $2$-form, the Liouville form. In this section, we introduce an analogous invariant form for algebraic billiards, and we calculate its divisor (its zeros and poles).

As before, let $C_0$ be a plane curve satisfying the conditions of Definition \ref{def_b}, and let $C \to C_0$ be a desingularization.

\begin{defn} \label{def_invt_form}
    Let $f: V \vdash V$ be a dominant correspondence, defined by $(\Gamma_f, \pi_1, \pi_2)$. A differential form $\eta$ with rational coefficients is said to be \emph{$f$-invariant} if $\pi_1^* \eta = \pi_2^* \eta$.
\end{defn}

\begin{thm}[= Theorem \ref{thm_main_invariant_form}] \label{thm_invariant_form}
Let $\omega$ be the rational $2$-form on $X \times Q$ defined on $T \A^2 \subset X \times Q$ by the formula
$$\omega \colonequals dx_0 \wedge dq_0 + dx_1 \wedge dq_1.$$
We may pull back $\omega$ to a form $\omega|_{C \times D}$ defined on $C \times D$, via $C \to C_0 \subset X$ and $D \subset Q$. The $2$-form $\omega|_{C \times D}$ on $C \times D$ is invariant for the billiards correspondence $b$, the secant correspondence $s$, and the reflection correspondence $r$.
\end{thm}

Before proving the theorem, we comment on the naturality of this construction. The choice of coordinates $x_0, x_1$ on $\A^2$ induces cotangent vectors $q_0 = dx_0, q_1 = dx_1$. Define dual coordinates $\xi_0, \xi_1 : T^* \A^2 \to k$, by linearly extending the rules
$$\xi_0(q_0) = 1, \quad \xi_0(q_1) = 0,$$
$$\xi_1(q_0) = 0, \quad \xi_1(q_1) = 1.$$
The differential $2$-form $dx_0 \wedge d\xi_0 + dx_1 \wedge d\xi_1$ is a well-known object called the \emph{canonical form} on $T^* \A^2$. It is independent of the initial coordinate choice $x_0, x_1$; see \cite[Section 2.2]{MR1853077}.

The canonical form is defined on a cotangent space, whereas algebraic billiards is defined on a tangent space; but using the quadratic form $q_0^2 + q_1^2$, we can view the canonical form as a form on $T\A^2$. (Since $\charac k \neq 2$, any choice of non-degenerate quadratic form gives a symmetric bilinear pairing on vectors, hence an identification of covectors with vectors.) Our choice of $q_0^2 + q_1^2$ identifies the vector $(q_0, q_1)$ with the covector $(\xi_0, \xi_1)$ and hence identifies the canonical form with $\omega$.

We prove Theorem \ref{thm_invariant_form} by elementary computations in the sheaf of differential forms. However, The idea behind the computation is, in some sense, the symplectic reduction argument of Khesin-Tabachnikov for pseudo-Riemannian billiards \cite[Section 3.3]{MR2518642}. They credit the underlying principle to Melrose \cite{MR436225, MR614395}.

\begin{proof}[Proof of Theorem \ref{thm_invariant_form}]
    Since $b = r \circ s$, it suffices to prove that $\omega|_{C \times D}$ is $s$-invariant and $r$-invariant. Further, we need only check the claim on a Zariski dense subset of $C \times D$. The definition of algebraic billiard table (Definition \ref{def_b}) excludes the possibility that the line at infinity is an irreducible component of $C$, and also excludes the possibility of a component of $C$ being an isotropic line. So we can just prove the claim on a Zariski dense subset of $C \times D$ above $T \A^2$. Thus we may work in affine coordinates. Our coordinates on $(T \A^2)^2$ are $(x_0, x_1, q_0, q_1)$ and $(x'_0, x'_1, q'_0, q'_1)$. 
    
    \par(1) Proof of $s$-invariance. Let $S$ be the relation of Definition \ref{def_secant_relation}, and let $S|_{X \times D}$ be its restriction to $(X \times D)^2$. We claim that $\omega|_{X \times D}$ is $S|_{X \times D}$-invariant; it follows that $\omega|_{C \times D}$ is $s$-invariant, by restriction.

    Let 
    $$\pi_1, \pi_2 : \Gamma_{S|_{X \times D}} \to X \times D$$
    be the two projections, and define
    $$\tilde{\omega} = \pi_1^* \omega - \pi_2^* \omega.$$
    We claim that $\tilde{\omega} = 0$.

    For notational convenience, identify $S|_{X \times D}$ with its graph. On $S|_{X \times D}$, the following affine equations hold:
    \begin{eqnarray}
    q_0^2 + q_1^2 = 1, \\
    q_0 = q'_0, \\
    q_1 = q'_1, \\
    \begin{vmatrix}
    x_0 & x'_0 & q_0 \\
    x_1 & x'_1 & q_1 \\
    1 & 1 & 0
    \end{vmatrix} = 0. \label{eq_det_S}
\end{eqnarray}

In coordinates, we have
\begin{align}
    \tilde{\omega} &= dx_0 \wedge dq_0 + dx_1 \wedge dq_1 - dx'_0 \wedge dq'_0 - dx'_1 \wedge dq'_1 & \text{(by definition)} \nonumber \\
    &= dx_0 \wedge dq_0 + dx_1 \wedge dq_1 - dx'_0 \wedge dq_0 - dx'_1 \wedge dq_1 & \text{(since } (q_0, q_1) = (q'_0, q'_1) \text{)} \nonumber \\
    &= (dx_0 - dx'_0) \wedge dq_0 + (dx_1 - dx'_1) \wedge dq_1 & \text{(rearranging).} \label{eq_nice_omega_S}
\end{align}

Taking the exterior derivative of the defining equation of $D$, and remembering that $\charac k \neq 2$, we get
\begin{equation} \label{eq:qdq}
    q_0 dq_0 + q_1 dq_1 = 0.
\end{equation}
Writing out \eqref{eq_det_S}, we have
$$0 = q_0 ( x_1 - x'_1 ) - q_1 (x_0 - x'_0).$$
Taking the exterior derivative gives a relation in the sheaf of differential 1-forms on $S|_{X \times D}$:
\begin{align*}
    0 &= d(q_0 ( x_1 - x'_1 ) - q_1 (x_0 - x'_0)) \\
    &= (x_1 - x'_1) dq_0 + q_0 (dx_1 - dx'_1) - (x_0 - x'_0) dq_1 - q_1 (dx_0 - dx'_0) & \text{(Leibniz rule)}.
\end{align*}
Wedging with $dq_0$ gives the following relation in the sheaf of 2-forms on $S|_{X \times D}$:
\begin{align*}
    0 &= ((x_1 - x'_1) dq_0 + q_0 (dx_1 - dx'_1) \\ & \quad - (x_0 - x'_0) dq_1 - q_1 (dx_0 - dx'_0)) \wedge dq_0  &\\
    &= q_0 (dx_1 - dx'_1) \wedge dq_0 - q_1 (dx_0 - dx'_0) \wedge dq_0 & \text{(2-forms on $D$ are trivial)} \\
    &= -q_1 (dx_1 - dx'_1) \wedge dq_1 - q_1 (dx_0 - dx'_0) \wedge dq_0 & \text{(using \eqref{eq:qdq})} \\
    &=- q_1 ((dx_1 - dx'_1) \wedge dq_1 + (dx_0 - dx'_0) \wedge dq_0) & \text{(factoring)} \\
    &= -q_1 \tilde{\omega} & \text{(by \eqref{eq_nice_omega_S})}.
\end{align*}

Since $-q_1$ is not identically $0$ on $D$, we obtain $0 = \tilde{\omega}$.

\par(2) Proof of $r$-invariance. First, assume $C$ is reduced; the general argument may then be done componentwise, since the claim is local. By a change of coordinates in $O_2(k)$, we may assume that $C$ has no lines of slope $0$. Let $R|_{C \times Q}$ be the relation of Definition \ref{def_reflection_relation}. We claim that $\omega|_{C \times Q}$ is $R|_{C \times Q}$-invariant. Let
$$\pi_1, \pi_2 : \Gamma_{R|_{C \times Q}} \to C \times Q$$
be the two projections, and define
$$\tilde{\omega} = \pi_1^* \omega - \pi_2^* \omega.$$
We claim that $\tilde{\omega} = 0$.

Let $C_0 \subset \PP^2$ be cut out by the affine equation $F(x_0, x_1) = 0$. Then, on $R|_{C \times Q}$, the following affine equations hold.
    \begin{eqnarray}
    F(x_0, x_1) = 0, \label{eq_C} \\
    x_0 = x'_0, \nonumber \\
    x_1 = x'_1, \nonumber \\
    \begin{vmatrix}
    q_0 & q'_0 & n(x)_0 \\
    q_1 & q'_1 & n(x)_1 \\
    1 & 1 & 0
    \end{vmatrix} = 0. \label{eq_det_R}
\end{eqnarray}
In coordinates, we have
\begin{align}
    \tilde{\omega} &= dx_0 \wedge dq_0 + dx_1 \wedge dq_1 - dx'_0 \wedge dq'_0 - dx'_1 \wedge dq'_1 & \text{(by definition)} \nonumber \\
    &= dx_0 \wedge dq_0 + dx_1 \wedge dq_1 + dx_0 \wedge dq'_0 + dx_1 \wedge dq'_1 & \text{(since } (x_0, x_1) = (x'_0, x'_1) \text{)} \nonumber \\
    &= dx_0 \wedge (dq_0 - dq'_0) + dx_1 \wedge (dq_1 - dq'_1) & \text{(rearranging).} \label{eq_nice_omega_R}
\end{align}
Since $C$ is reduced, the exterior derivative of \eqref{eq_C} gives us
\begin{equation}
    n(x)_0 dx_0 + n(x)_1 dx_1 = 0.
    \label{eq_dF}
\end{equation}
Writing out \eqref{eq_det_R}, we have
$$ 0 = n(x)_0 (q_1 - q'_1) - n(x)_1 (q_0 - q'_0).$$
The exterior derivative is
\begin{align*}
    0 &= d(n(x)_0 (q_1 - q'_1) - n(x)_1 (q_0 - q'_0)) \\
    &= d(n(x)_0) (q_1 - q'_1) + n(x)_0 \wedge (dq_1 - dq'_1) - d(n(x)_1) (q_0 - q'_0) - n(x)_1 (dq_0 - dq'_0) & \text{(Leibniz)}.
\end{align*}
Wedging with $dx_0$ gives the following equation:
\begin{align*}
0 &= ( d(n(x)_0) (q_1 - q'_1) + n(x)_0 \wedge (dq_1 - dq'_1) \\
& \quad - d(n(x)_1) (q_0 - q'_0) - n(x)_1 (dq_0 - dq'_0) ) \wedge dx_0 \\
& = n(x)_0 (dq_1 - dq'_1) \wedge dx_0 - n(x)_1 (dq_0 - dq'_0) \wedge dx_0  &  \text{($2$-forms on $C$ are trivial)} \\
& = -n(x)_1 (dq_1 - dq'_1) \wedge dx_1 - n(x)_1 (dq_0 - dq'_0) \wedge dx_0 & \text{(using \eqref{eq_dF})} \\
&= -n(x)_1 \tilde{\omega} & \text{(by \eqref{eq_nice_omega_R})}.
\end{align*}
Recalling that $C_0$ contains no lines of slope $0$, we may divide by $-n(x)_1$, so $0 = \tilde{\omega}$.
\end{proof}

\begin{remark} \label{rem_no_invt_form_on_CP1}
    The secant and reflection correspondences commute with the $2$-to-$1$ map 
    \begin{align*}
        C \times D &\to C \times \PP^1,\\
        (c, q) &\mapsto (c, [q]).
    \end{align*}
    Thus there is a semiconjugacy from $b$, defined in Definition \ref{def_b}, to a more well-studied correspondence $b_{C \times \PP^1}$ on $C \times \PP^1$, as explained in Proposition \ref{prop_b_basic} (3).

    The map $b_{C \times \PP^1}$ generally does not preserve any $2$-form besides $0$. Indeed, suppose by way of contradiction that $\omega_0$ is an invariant $2$-form of $b_{C \times \PP^1}$. We will show that $\omega_0 = 0$.
    
    The pullback of $\omega_0$ to $C \times D$ is preserved by $b$. Since $\omega \neq 0$, the quotient $\omega_0 / \omega$ is a well-defined invariant rational function of $b$, which gives rise to an invariant rational function of $b_{C \times \PP^1}$. Invariant rational functions for billiards do not generally exist, except for constant functions \cite{MR4210728}. Thus $\omega_0 / \omega$ is constant.
    Now, recall the map $\mu$ defined in the proof of Proposition \ref{prop_b_basic}. It is easy to check in affine coordinates that
    $$\mu^* \omega = -\omega.$$
    We also have $\mu^* \omega_0 = \omega_0$, since $\omega_0$ was defined by pulling back a form on the quotient of $C \times D$ by the involution $\mu$. Thus
    \begin{align*}
        \omega_0 / \omega &= \mu^*(\omega_0 / \omega) & \text{(pulling back a constant function)} \\
        &= \mu^*\omega_0 / \mu^* \omega \\
        &= -\omega_0 / \omega.
    \end{align*}
    Since $\charac k \neq 2$, we then have $\omega_0 = 0$.
    
    The situation therefore resembles that of monomial maps and Chebyshev maps on $\PP^1$. Chebyshev maps do not preserve a form, but they are the image by a $2$-to-$1$ semiconjugacy of a monomial map (see e.g. \cite{Silverman10}) preserving the $1$-form $dx/x$.
\end{remark}

\subsection{The divisor of the invariant form} \label{sect_form_divisor}

We denote the divisor of a form $\eta$ by $\dv \eta$. The \emph{Gauss embedding} is defined as 
\begin{align*}
    G & \colon C \to C \times \PP^1,\\
    & c \mapsto (c, t(c)).
\end{align*}
It is an isomorphism onto its image $G(C)$. Define a curve $\hat{G}(C) \subset C \times D$ by
$$\hat{G}(C) = \{ (c, q) : (c, [q]) \in G(C) \}.$$
Then $\hat{G}(C)$ is the fiber product over $\PP^1$ of the tangent slope map $t: C \to \PP^1$ and the slope map $[ \cdot ]: D \to \PP^1$. Note that $\hat{G}(C)$ is a ramified double cover of $C$.

\begin{prop} \label{prop_div}
Let the curve $C$ be general. The divisor of $\omega$ is
\begin{equation}
    \dv \omega = \hat{G}(C) + \sum_{c_i \in C_\infty} (d-3)(c_i \times D) + \sum_{q_j \in D_\infty} (-1) (C \times q_j). \label{eq_div}
\end{equation}
\end{prop}

\begin{proof}
    Suppose $V$ is an irreducible curve not of the form $c_\infty \times D$ or $C \times q_\infty$. We will call any such $V$ a \emph{main affine curve}, since $V$ has an affine equation in the coordinates $x_0, x_1, q_0, q_1$. We break up the divisor as the following sum, for some multiplicities $\mu_i, \nu_j$:
        $$\dv \omega = \sum_{\text{main affine }V} \ord_V \omega + \sum_{c_i \in C_\infty} \mu_i (c_i \times D) + \sum_{q_j \in D_\infty} \nu_j (C \times q_j).$$
    Now we compute terms (1), (2), and (3).
    
    Let the partial derivatives of $F$ with respect to $x_0$ and $x_1$ be denoted $F_{x_0}$, $F_{x_1}$.
    By the defining relations in the sheaf of differential forms, we have
    $$ \omega = \dfrac{F_{x_0} q_0 + F_{x_1} q_1}{F_{x_0} q_0} dx_1 \wedge dq_1.$$
    So
    $$\dv \omega = \dv(F_{x_0} q_0 + F_{x_1} q_1) + \dv \dfrac{dx_1 \wedge dq_1}{F_{x_0} q_0}.$$
    \par(1) $= \hat{G}(C)$. \enspace
    
    We claim that if $V$ is a main affine curve, we have
    $$\ord_V (-\dv(F_{x_0}) - \dv(q_0) + \dv(dx_1 \wedge dq_1)) = 0.$$
    The form $(dx_1 / F_{x_0}) \wedge (dq_1 / q_1)$ is a wedge of a $1$-form on $C$ and a $1$-form on $D$), so its divisor can be computed as
    $$\dv_{C \times D} \left((dx_1 / F_{x_0}) \wedge (dq_1 / q_1)\right) = \dv_C (dx_1 / F_{x_0}) \times D + C \times \dv_D (dq_1 / q_1).$$
    Since $\charac k \neq 2$ and $C$ has simple branching over the $x_1$-line (by generality), for all $c \in C \smallsetminus C_\infty$ and $q \in D \smallsetminus D_\infty$, we have
    $$\ord_c F_{x_0} = \ord_c \dv_C dx_1,$$
    $$\ord_q q_0 = \ord_q \dv_D dq_1.$$
    It follows that
    $$\sum_{\text{main affine }V} \ord_V \omega = \sum_{\text{main affine V}} \ord_V F_{x_0} q_0 + F_{x_1} q_1.$$

    Since $C$ is smooth by generality, the equation $F_{x_0} q_0 + F_{x_1} q_1 = 0$ is a defining affine equation of $\hat{G}(C)$. We claim that $\hat{G}(C)$ is integral; as a consequence, term (1) is just $\hat{G}(C)$.
    
    To see this, recall that $\hat{G}(C)$ is the fiber product over $\PP^1$ of the tangent slope map $t: C \to \PP^1$ and the slope map $[ \cdot ]: D \to \PP^1$. Each ramification point of $t$ is in $C_\infty$ or is an inflection point. Since $C$ is general, we may assume that $\pm i$ are not branch points of $t$. Since $t$ and $[ \cdot ]$ are maps of smooth curves to $\PP^1$ and have distinct branch loci, the fiber product $\hat{G}(C)$ over these maps is smooth. Then $\hat{G}(C) \to C$ is a $2$-to-$1$ cover of smooth curves ramified at isotropic tangency points of $C$, so $\hat{G}(C)$ is irreducible.

    \par(2): For all $i$, we have $\mu_i = d - 3$. To see this, first notice that the order of $F_{x_0} q_0 + F_{x_1} q_1$ along any curve $c_i \times D$ is $0$. Then, using again the split decomposition from (1), we just need to show that on $C$, we have 
    $$\ord_{c_i} (dx_1 / F_{x_0}) = d - 3.$$
    The degree of $dx_1 / F_{x_0}$ on $C$ is $2g(C) - 2 = d(d-3)$, and that form has no zeros or poles on $C$ except possibly on $C_\infty$, as argued in (1). Since $C$ is general, the order of $dx_1 / F_{x_0}$ at each $c_i$ is the same, and there are $d$ of these points, giving an order of $d - 3$ at each.
    
    \par(3): For all $j$, we have $\nu_j = -1$.
    This case is exactly as in (2) with roles of $C$ and $D$ reversed.
\end{proof}

\begin{remark}
If $C$ is the circle $x_0^2 + x_1^2 = 1$, one can show that $\hat{G}(C)$ is a union of two rational curves meeting at two points.
\end{remark}

\section{The Cheap Dynamical Degree Bound} \label{sect_cheap}

For any rational map $\phi$ of projective space, the first dynamical degree satisfies the basic inequality
\begin{equation} \label{eq_dd_ineq_proj}
    \lambda_1(\phi) \leq \deg \phi.
\end{equation}
For correspondences on arbitrary smooth surfaces, the same type of inequality holds. The algebraic degree $\deg \phi$ generalizes to the spectral radius of the action of $\phi$ on the group of divisor classes up to numerical equivalence. In this section, we compute these upper bounds for the billiards correspondence 
$$b : C \times D \vdash C \times D$$
on a general curve $C$, proving a ``cheap'' version of Theorem \ref{thm_main}. Besides being instructive, our arguments here are ingredients for the full proof of Theorem \ref{thm_main} in Section \ref{sect_global_bound}.

\begin{lemma} \label{lemma_rh}
Let $C \subset \PP^2$ be a smooth curve of degree $d \geq 2$, and let $t \in \PP^1$, where $\PP^1$ is identified with the line at infinity. If $t \not\in C$, then $C$ has $d(d-1)$ points with tangent line of slope $t$, counted with multiplicity.
\end{lemma}

\begin{proof}
    The map $w: C \to \PP^1$ defined by the pencil of lines through $t$ is ramified exactly at the tangents of $C$ with slope $t$. Since $t \not\in C$, the map $w$ has degree $d$. By smoothness, the genus of $C$ is $\frac{1}{2}(d-1)(d-2)$. The Riemann-Hurwitz formula then gives $d(d-1)$ ramification points of $w$ counted with multiplicity, corresponding to tangents of slope $t$.
\end{proof}

\begin{thm}[Cheap bound] \label{thm_cheap_upper_bound}
    Let $C$ be a smooth degree $d \geq 2$ curve in $\PP^2$ such that $C_\infty$ does not contain $[1 : i : 0]$ or $[1 : -i : 0]$. Let $b : C \times D \vdash C \times D$ be the billiards correspondence. Then
    $$\lambda_1( b ) \leq d^2 + \sqrt{d^4 - 5d^2 + 4d + 1} < 2d^2.$$
\end{thm}

\begin{proof}
    Choose general points $c_0 \in C$ and $q_0 \in D$, and define divisors 
    $$C^{(0)} = C \times q_0, \quad D^{(0)} = c_0 \times D.$$
    The curve $C^{(0)}$ is smooth, hence irreducible, and $D^{(0)}$ is irreducible since $\charac k \neq 2$. Since $D \cong \PP^1$, the classes of $C^{(0)}$ and $D^{(0)}$ in $\Num^1( C \times D )$ form a basis for $\Num^1( C \times D ) \cong \Z^2$.

    Recall that $b = r \circ s$ by definition, where $s$ is the secant correspondence and $r$ is the reflection correspondence. Let $M_b$, $M_r$, and $M_s$ be the matrices of $b_*$, $r_*$, and $s_*$ in this basis.
    
    From Proposition \ref{prop_basic_s} and Proposition \ref{prop_basic_r}, we have a complete description of the contractions of $s$ and the expansions of $r$. The images of the contractions of $s$ all have $C$-coordinate in $C_\infty$. By the hypothesis that $C_\infty$ does not contain $[1 : \pm i : 0]$, the points of $\Ind r$ have $C$-coordinate outside $C_\infty$. So, by Theorem \ref{thm_dd_props} (\ref{it_functorial}), we have 
    $$b_* = r_* s_*.$$
    Applying Theorem \ref{thm_dd_props} (\ref{it_dd_bound}), we have
    \[
    \lambda_1 (b) \leq \rad M_b .
    \]
    Therefore it suffices to calculate $M_s$ and $M_r$ and to estimate $\rad M_r M_s$.

    The intersection form on $\Num^1 (C \times D)$ has matrix
    \[
    J \colonequals \begin{pmatrix}
        0 & 1 \\
        1 & 0 \\
    \end{pmatrix}.
    \]

    We claim that
    \begin{equation} \label{eq_cheap_s_matrix}
M_s = \begin{pmatrix}
d-1 & 2 \\
0 & d-1 \\
\end{pmatrix}.
    \end{equation}
    Let 
    $$\Pi_1, \Pi_2 : (C \times D)^2 \to C \times D$$
    be the two projections.
    As in \eqref{eq_pfwd_ix}, for any two divisors $\Delta, \Delta'$, we define a quantity $\iota_s(\Delta, \Delta')$ by 
    $$\iota_s(\Delta, \Delta') = s_* \Delta \cdot \Delta' = \Gamma_s \cdot \Pi_1^* \Delta \cdot \Pi_2^* \Delta'.$$
    Define a matrix
    \[
    M'_s \colonequals \begin{pmatrix}
        \iota_s( C^{(0)}, C^{(0)}) & \iota_s( D^{(0)}, C^{(0)}) \\
        \iota_s( C^{(0)}, D^{(0)}) & \iota_s(D^{(0)}, D^{(0)})
    \end{pmatrix}.
    \]
    The dual basis for $C^{(0)}$, $D^{(0)}$ is given by intersecting against $D^{(0)}$, $C^{(0)}$ respectively, so $M_s = J M'_s$.
    So it suffices to compute the entries of $M'_s$. 
    \begin{itemize}[leftmargin=*]
        \item $\iota_s (C^{(0)}, C^{(0)}) = 0$. Let $q_1 \in D$ be general, chosen so that $q_1 \neq q_0$. Since $C^{(0)} = C \times q_0$ and $C^{(1)} \colonequals C \times q_1$ are numerically equivalent, we have
        \begin{align*}
\Gamma_s \cdot \Pi_1^* C^{(0)} \cdot \Pi_2^* C^{(0)} &= \Gamma_s \cdot \Pi_1^* C^{(0)} \cdot \Pi_2^* C^{(1)}.
        \end{align*}
        The latter intersection is empty, since the projections of $\Pi_1^* C^{(0)}$ and $\Pi_2^* C^{(1)}$ to $D$ are disjoint.
        \item $\iota_s (C^{(0)}, D^{(0)}) = d - 1$. Since $\Pi^*_1 C^{(0)}$ and $\Pi^*_2 D^{(0)}$ are smooth, the intersection $\Gamma_s \cdot \Pi_1^* C^{(0)} \cdot \Pi_2^* D^{(0)}$ agrees with the degree of the $0$-dimensional scheme $\Gamma_s|_{C^{(0)} \times D^{(0)}}$. This intersection is the number of pairs $c \in C, q \in D$ such that $(c, q_0, c_0, q) \in \Gamma_s$, counted with multiplicity. The definition of $s$ implies that any such quadruple satisfies $q_0 = q$, so the count is given by the number of intersections of $C$ with a line through $c_0$ with direction $q_0$, besides $c_0$ itself. A line intersects a degree $d$ curve in $d$ points, so there are $d - 1$ intersections.
        \item $\iota_s (D^{(0)}, C^{(0)}) = d - 1$, by the previous bullet and self-adjointness of $s$.
        \item $\iota_s (D^{(0)}, D^{(0)}) = 2$. Again by smoothness of the divisors, we need only compute the degree of $\Gamma_s|_{D^{(0)} \times D^{(0)}}$. All curves of the form $c \times D$, $c \in C$, are numerically equivalent. Let $D^{(1)} = c_1 \times D$, where $c_1 \neq c_0$ is a general point of $C$. The intersection $\Gamma_s \cdot \Pi_1^* D^{(0)} \cdot \Pi_2^* D^{(1)} $ is the number of values of $q \in D$ such that $[q]$ is the slope of the line $\overline{c_0 c_1}$. There are $2$ such values.
    \end{itemize}
    We have proved \eqref{eq_cheap_s_matrix}. Now we show
    \begin{equation} \label{eq_cheap_r_matrix}
    M_r = \begin{pmatrix}
1 & 0 \\
d(d-1) & 1 \\
\end{pmatrix}.
\end{equation}
Again we let $\Pi_1, \Pi_2 : (C \times D)^2 \to C \times D$ be the projections, and define $\iota_r (\Delta, \Delta')$ as before.
    \begin{itemize}[leftmargin=*]
        \item $\iota_r (C^{(0)}, C^{(0)}) = d(d-1)$. The intersection $\Gamma_r \cdot \Pi_1^* C^{(0)} \cdot \Pi_2^* C^{(0)}$ is the number of points $c \in C$ such that $r( c, q_0 ) = (c, q_0)$, with multiplicity. These are exactly the points $c \in C$ with a tangent line of slope $[q_0]$. By generality of $q_0$, we may assume that $C$ has no points at infinity with tangent slope $[q_0]$. By Lemma \ref{lemma_rh}, there are $d(d-1)$ such tangents.
        \item $\iota_r (C^{(0)}, D^{(0)}) = 1$. This intersection counts pairs $(c, q)$ such that $(c, q_0, c_0, q) \in \Gamma_r$. By the definition of $r$, we must have $c = c_0$; then there is only one value of $q$ possible, since $r$ is a rational map and generality of $c_0, q_0$ implies that $(c, q_0)$ is not an indeterminacy point of $r$.
        \item $\iota_r (D^{(0)}, C^{(0)}) = 1$, by the previous bullet and by self-adjointness of $r$.
        \item $\iota_r (D^{(0)}, D^{(0)}) = 0$. Let $c_1 \neq c_0$ be a general point of $C$, and let $D^{(1)} = c_1 \times D$. Then $D^{(0)}$ and $D^{(1)}$ are numerically equivalent. So
        $$\Gamma_r \cdot \Pi_1^*D^{(0)} \cdot \Pi_2^* D^{(0)} = \Gamma_r \cdot \Pi_1^* D^{(0)} \cdot \Pi_2^* D^{(1)}.$$
        Since $c_0 \neq c_1$, the divisors $\Pi_1^* D^{(0)}$ and $\Pi_2 D^{(1)}$ are disjoint, so the intersection is $0$.
    \end{itemize}
    
    Combining \eqref{eq_cheap_s_matrix} and \eqref{eq_cheap_r_matrix} gives
    \[ 
    \rad(M_b) = \rad \begin{pmatrix}
        d - 1 & 2 \\
        d(d-1)^2 & (2d + 1)(d - 1) \\
    \end{pmatrix}.
    \]
    The eigenvalues of $M_b$ are thus
    $$d^2 \pm \sqrt{d^4 - 5d^2 + 4d + 1}.$$
    Since we assumed $d \geq 2$, we have 
    $\rad( M_b ) < 2d^2.$
\end{proof}

In the case $d = 2$, which includes the complexification of the real elliptic billiard, this argument shows that $\lambda_1(b) \leq \frac{4 + \sqrt{5}}{2}$. But complete integrability of the elliptic billiard implies that $\lambda_1(b) = 1$. Our next step is therefore to look for a birational model on which we can obtain a tighter bound.

\section{Blowing up Scratch Points} \label{sect_local}
This section contains essential lemmas for the proofs of our main results, Theorem \ref{thm_main} (the dynamical degree bound) and Theorem \ref{thm_main_sc} (singularity confinement).

Let $C$ be an algebraic billiard curve of degree $d \geq 2$ (for instance, any smooth curve). We defined billiards $b$ on $C \times D$ in Definition \ref{def_b} as the composition of the secant correspondence $s$ and reflection correspondence $r$. We were able to compute a bound on the dynamical degree of $b$ (Section \ref{sect_cheap}), but we also found evidence that this bound was not sharp. In fact, this phenomenon has a geometric interpretation that is already manifest over $\R$, as we now explain.

We classified the indeterminacy points of $s$ and $r$ in Proposition \ref{prop_basic_s} and Proposition \ref{prop_basic_r}. It will be helpful to dub these points \emph{scratch points}, as follows.

\begin{defn} \label{def_scratch}
A \emph{scratch point at infinity} is a point $p = (c, q) \in C \times D$ such that $c \in C_\infty$ and $[q] = t(c)$. A scratch point at infinity is \emph{basic} if $q$ is non-isotropic (i.e. $q \not\in D_\infty$) and $c$ is a nonsingular point of $C$ with a tangent line of least order (quadratic), transverse to the line at infinity in $X$.

An \emph{isotropic scratch point} is a point $p = (c, q) \in C \times D$ such that $q \in D_\infty$ and $[q] = t(c)$. An isotropic scratch point is \emph{basic} if $c \not\in C_\infty$ and $c$ is a nonsingular point of $C$ with a tangent line of least order (quadratic).
\end{defn}

We can do better than the cheap bound of Theorem \ref{thm_cheap_upper_bound} by considering ``destabilizing orbits'', or as they are called in the integrable systems literature, ``confined mapping singularities''. Informally, if $p = (c, q)$ is a basic scratch point at infinity, then there is a singularity pattern
$$ (C \times q)\smallsetminus p \quad \overset{s}{\mapsto} \quad p \quad \overset{r}{\mapsto} \quad p \quad \overset{s}{\mapsto} \quad (C \times q) \smallsetminus p.$$
There is also a singularity pattern for each basic isotropic scratch point $p$:
$$(c \times D) \smallsetminus p \quad \overset{r}{\mapsto} \quad p \quad \overset{s}{\mapsto} \quad p \quad \overset{r}{\mapsto} \quad (c \times D) \smallsetminus p. $$

It turns out that blowing up $p$ allows us to extend the domain in a way that resolves these singularity patterns. For instance, if $p$ is a basic scratch point at infinity, we show that it is possible to compute a well-defined image of a point in $(C \times q) \smallsetminus p$ by $s \circ r \circ s$, as depicted in Figure \ref{fig_sc}.

\begin{figure}
\includegraphics[width=6.1in]{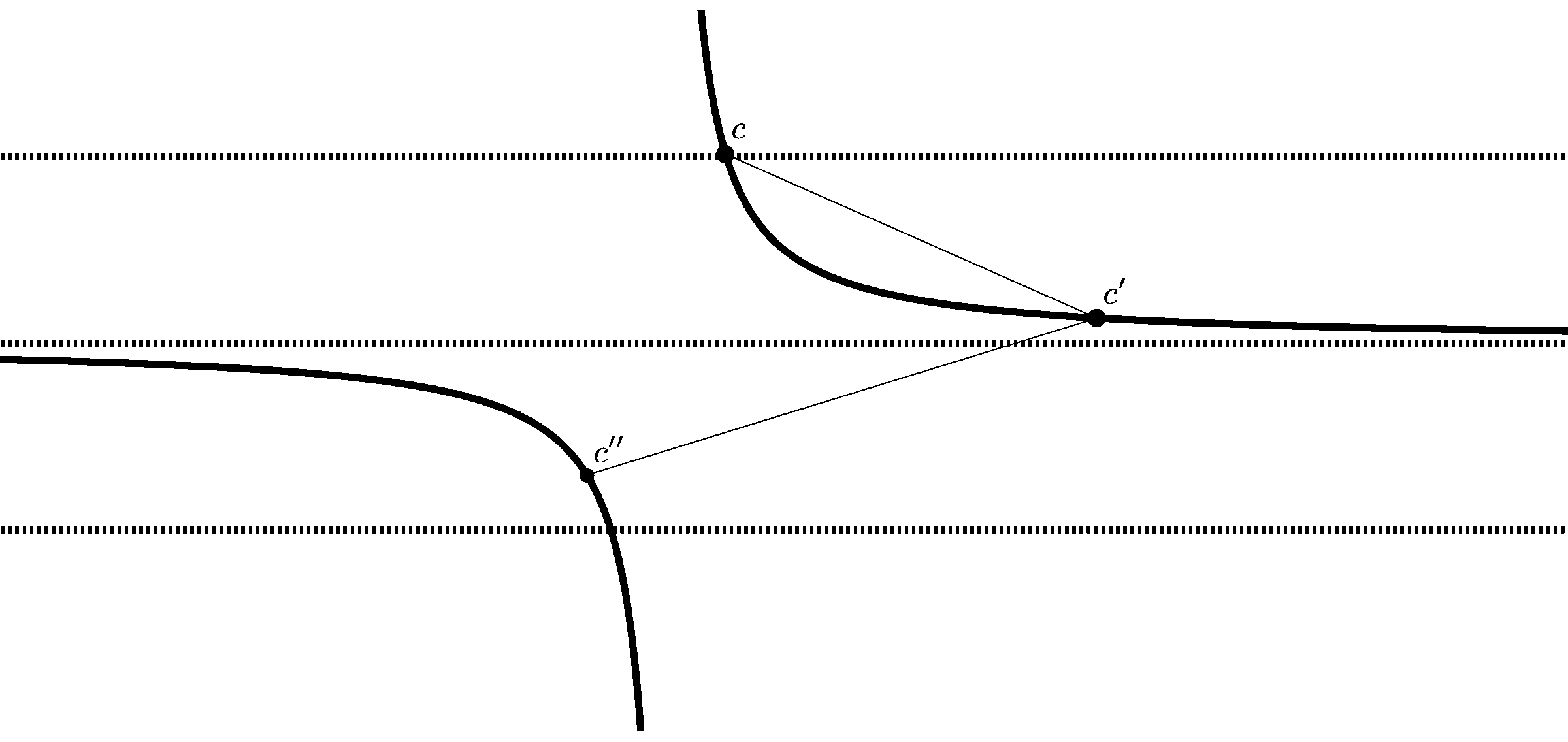}
\caption{Singularity confinement for billiards. Here we show a hyperbola $C$ with a horizontal asymptote, corresponding to a point $c_\infty = [1 : 0 : 0]$. Let $c \in C$. The three dotted lines are the asymptote, the horizontal line $\ell$ through $c$, and the reflection of $\ell$ across the asymptote. As the direction of the first shot tends to $[1 : 0 : 1]$, the limiting position of $c''$ is the intersection of the lower dotted line with $C$. Thus the limit of $c''$ depends on $c$, but the limit of $c'$ does not.}
    \label{fig_sc}
\end{figure}

We present these lemmas together to highlight their similarities, then give their proofs in the next section. The proofs are technical and involve rather delicate computations, and can be safely skipped on first reading.

Our first lemma computes the strict transform of the secant correspondence when we blow up a basic scratch point at infinity.

\begin{lemma}[Secant at infinity] \label{lemma_secant_at_infty}
Let $s: C \times D \vdash C \times D$ be the secant correspondence. Let $p = (c, q) \in C \times D$ be a basic scratch point at infinity, and let $P \to C \times D$ be the blowup at $p$, with exceptional divisor $E$. Let $s_{\str}: P \vdash P$ be the strict transform of $s$.
\begin{enumerate}
    \item We may identify $E$ with the pencil of lines through $c$ in $X = \PP^2$. Let $T_c$ be the tangent line at $c$. 
    Given $e \in E$ corresponding to a general line $\ell$, the image set is 
    $$s_{\str}(e) = (\ell \cap C \smallsetminus c) \times q.$$ 
    The point $e \in E$ corresponding to $T_c$ has image set $s_{\str}(e) = (\ell \cap C \smallsetminus c) \cup \{e\}.$
    \item As divisors,
    $$ {s_{\str}}_* E = (C \times q)_{\str}.$$
\end{enumerate}
\end{lemma}

Our next lemma describes the behavior of the strict transform of $s$ when we blow up a basic isotropic scratch point.

\begin{lemma}[Secant at isotropic scratch points] \label{lemma_secant_at_isotropic}
Let $s: C \times D \vdash C \times D$ be the secant correspondence. Let $p = (c, q) \in C \times D$ be a basic isotropic scratch point, and let $P \to C \times D$ be the blowup at $p$, with exceptional divisor $E$. Let $s_{\str} : P \vdash P$ be the strict transform.
\begin{enumerate}
    \item Given $e \in E$, the image set ${s_{\str}}(e)$ consists of one point of $E$ together with the multiset 
    $$(T_c \cap C \smallsetminus \{c\}) \times q,$$
    where $T_c$ is the tangent to $C$ at $c$. 
    \item As divisors,
    $${s_{\str}}_* E =  E.$$
\end{enumerate}
\end{lemma}

Due to the near-symmetry between the definitions of secant and reflection (Remark \ref{rem_near_duality}), our lemmas for reflection are essentially identical to those for secant. The following lemma is the counterpart of Lemma \ref{lemma_secant_at_infty}.

\begin{lemma}[Reflection at isotropic scratch points] \label{lemma_reflect_isotropic}
Let $r: C \times D \vdash C \times D$ be the reflection correspondence. Let $p \in C \times D$ be a basic isotropic scratch point $(c, I)$, where $I = \pm i$. Let $P \to C \times D$ be the blowup at $p$, with exceptional divisor $E$. Let $r_{\str}: P \vdash P$ be the strict transform of $r$ to $P$. 
\begin{enumerate}
    \item Given $e \in E$, the image set $r_{\str}(e)$ consists of one point in $(c \times D)_{\str}$.
    \item As divisors,
    $${r_{\str}}_* E = (c \times D)_{\str}.$$
\end{enumerate}
\end{lemma}

The following lemma is the counterpart of Lemma \ref{lemma_secant_at_isotropic}.

\begin{lemma}[Reflection at infinity] \label{lemma_reflect_infty}
Let $r: C\times D \vdash C \times D$ be the reflection correspondence. Let $p = (c, q) \in C \times D$ be a basic scratch point at infinity, and let $P \to C \times D$ be the blowup at $p$, with exceptional divisor $E$. Let $r_{\str}: P \vdash P$ be the strict transform of $r$.
\begin{enumerate}
    \item The image $r_{\str}(e)$ of any point $e \in E$ is a point in $E$.
    \item As divisors,
    $${r_{\str}}_* E = E.$$
\end{enumerate}
\end{lemma}

\begin{remark}
    In the identification of Lemma \ref{lemma_secant_at_infty}, if $\ell$ is the line through $c$ corresponding to $e$, then $r_{\str}(e)$ corresponds to the reflection of $\ell$ across the tangent $T_c$. This can be proved in local coordinates.
\end{remark}

\subsection{Proofs of the blowup lemmas}
First we prove the easy ones.

\begin{proof}[Proof of Lemma \ref{lemma_secant_at_isotropic}]
Since $p$ is basic, we do not have $c \in C_\infty$, so $p$ is not in the indeterminacy locus of $s$. Therefore, for each pair of the form $(p, p') \in \Gamma_s$ where $p \neq p'$, the strict transform $s_{\str}$ has a contraction $E \times p'$. Again since $p$ is basic, there are $d - 2$ such points $p' = (c', q)$, where $c'$ ranges over intersections besides $c$ of the tangent line $T_c$ with $C$.

We claim that $s_{\str}$ induces a $1$-to-$1$ map $E \to E$; this finishes proving (1) and (2).

Since $p$ is basic, the correspondence $s$ is regular at $p$. Since $p 
\in s(p)$, $\hat{s}$ induces a self-correspondence of $E$ by $2$-sided restriction. The existence of $d-2$ points in $s(p)$ besides $p$ itself implies that $\hat{s}|_{E \times E}$ is single-valued. We must check that it's not a constant map. Restrict $s$ to the completion of $C\times D$ at $p$; this doesn't affect our local calculation. Recall that $E$ is naturally the projectivized tangent space $\PP T_p(C\times D)$. Since $s$ fixes $p$ and $s = s^{-1}$ in the completion, the differential of $s$ at $p$ is invertible, so $\hat{s}|_{E \times E}$ is invertible and in fact an involution. Hence $s(E)=E$ and $E$ has multiplicity $1$ in $s_* E$.
\end{proof}

\begin{proof}[Proof of Lemma \ref{lemma_reflect_infty}]
    Since $p$ is basic, the rational map $r$ is regular at $p$. Since $r(p)=p$, we have $\hat{r}(E) \subset E$, so $\hat{r}$ induces a self-map of $E$ by $2$-sided restriction. Recall that $E$ is naturally the projectivized tangent space $\PP T_p(C\times D)$. Since $r$ fixes $p$ and $r = r^{-1}$, the differential of $r$ at $p$ is invertible, so $\hat{r}|_{E \times E}$ is invertible and in fact an involution. Hence $r(E)=E$ and $E$ has multiplicity $1$ in $r_* E$.
\end{proof}

The remaining two are trickier.

\begin{proof}[Proof of Lemma \ref{lemma_secant_at_infty}]
Step 1: Reduction. First, we reduce to the case that $c = [1: 0 : 0]$ and the tangent line to $c$ is $X_1 = 0$. To see that this suffices, note that affine transformations act on $X \times Q$. In particular, elements of $O_2(k)$ and translations preserve the standard Euclidean metric $dx_0^2 + dx_1^2$, hence also preserve the conic $D$. Since $p$ is basic, the point $c$ has a well-defined tangent line $T_c$ of slope $t_c \neq \pm i$. The group $O_2(k)$ acts transitively on non-isotropic directions, so there is an element of $O_2(k)$ taking lines of slope $t_c$ to lines of slope $[1:0]$. Then applying a translation allows us to move $T_c$ so that it passes through the origin.

Step 2: Defining charts. We use coordinates 
$$[X_0 : X_1 : X_2], [Q_0 : Q_1 : Q_2], \quad [X'_0 : X'_1 : X'_2], [Q'_0 : Q'_1 : Q'_2]$$
on the first and second factors of $\Gamma_s \subset (X \times Q)^2$.
We define an affine subset $A \times B \cong \A^4 \times \A^4$ of $(X \times Q)^2$ by
$$A \colonequals \Spec k[u, v, y, z], \quad B \colonequals \Spec k[x'_0, x'_1, y', z'],$$
where
$$u = X_1 / X_0, \quad v = X_2 / X_0, \quad y = Q_0/Q_2 - 1, \quad z = Q_1 / Q_2,$$
$$x'_0 = X'_0 / X'_2, \quad x'_1 = X'_1 / X'_2, \quad y' = Q'_0 / Q'_2 - 1, \quad z' = Q'_1 / Q'_2.$$
The affine equation of $C$, by our symmetry reduction above, is
\begin{equation}
    0 = u + O(u^2, uv, v^2). \label{eq_C_uv}
\end{equation}
The affine equation of $D$ is
\begin{equation}
    0 = y^2 + 2y + z^2. \label{eq_D_yz}
\end{equation}
Then $\Gamma_s|_{A \times B} \subset (C \times D)^2 \cap V$, where $V$ is the variety with ideal generated by
\begin{align*}
&\begin{vmatrix}
    1 & x'_0 & y + 1 \\
    u & x'_1 & z \\
    v & 1 & 0
\end{vmatrix},\\
&y - y',  \\
&z - z'.
\end{align*}

Step 3: Equations for $C, D, E$ in a one-sided blowup. The coordinates of $p$ in $A$ are $(0,0,0,0)$. Consider the blowup of $A$ at $p$. It has an affine subset 
$$\hat{A} = \Spec k[\hat{u}, \hat{v}, \hat{y}, \hat{z}]$$ 
defined by
$$ \hat{u} \hat{z} = u, \quad \hat{v} \hat{z} = v, \quad \hat{y} \hat{z} = y, \quad \hat{z} = z.$$
By \eqref{eq_C_uv}, \eqref{eq_D_yz}, the total transforms of $C$ and $D$ to $\hat{A}$ are defined respectively by
$$ 0 = \hat{u} \hat{z} + O(\hat{z}^2),$$
$$ 0 = \hat{y}^2 \hat{z}^2 + 2 \hat{y} \hat{z} + \hat{z}^2.$$
Therefore the strict transforms $C_{\str, 1}, D_{\str, 1}$ of $C$ and $D$ to $\hat{A}$ are defined respectively by
\begin{equation}
    0 = \hat{u} + O(\hat{z}), \label{eq_C_uv_str}
\end{equation}
\begin{equation}
    0 = \hat{y} \hat{z} + 2 \hat{y} + \hat{z}. \label{eq_D_yz_str}
\end{equation}
Recall that $E$ is the exceptional divisor. We claim that in $\hat{A}$, the line $E$ is parametrized by $\hat{v}$, and consists of points of the form $(0, \hat{v}, 0, 0)$. Indeed $\hat{z}(E) = 0$ by definition of the exceptional divisor, and it follows from \eqref{eq_C_uv_str}, \eqref{eq_D_yz_str} that $\hat{y}(E) = 0$ and $\hat{u}(E) = 0$.

Step 4: The strict transform of the secant graph. Let $V_{\tot, 1}$ denote the total transform of $V$ in $\hat{A} \times B$. The ideal of $V_{\tot, 1}$ is generated by
\begin{align}
& \begin{vmatrix}
    1 & x'_0 & \hat{y} \hat{z} + 1 \\
    \hat{u} \hat{z} & x'_1 & \hat{z} \\
    \hat{v} \hat{z} & 1 & 0
\end{vmatrix}, \label{eq_midway_det_of_s_inf} \\
& \hat{y} \hat{z} - y', \nonumber \\
& \hat{z} - z'. \nonumber
\end{align}
We can rewrite the determinant \eqref{eq_midway_det_of_s_inf} by subtracting the third column from the first, to obtain
\begin{equation}
\begin{vmatrix}
    -\hat{y} \hat{z} & x'_0 & \hat{y} \hat{z} + 1 \\
    (\hat{u} - 1) \hat{z} & x'_1 & \hat{z} \\
    \hat{v} \hat{z} & 1 & 0
\end{vmatrix}. \label{eq_midway_2_det_of_s_inf}
\end{equation}
We claim that the strict transform $V_{\str, 1}$ of $V$ to $\hat{A} \times B$ has ideal defined by dividing the first column of \eqref{eq_midway_2_det_of_s_inf} by $\hat{z}$, that is,
\begin{align*}
& 
\begin{vmatrix}
    -\hat{y} & x'_0 & \hat{y} \hat{z} + 1 \\
    \hat{u} - 1 & x'_1 & \hat{z} \\
    \hat{v} & 1 & 0
\end{vmatrix}, \\
& \hat{y} \hat{z} - y', \nonumber \\
& \hat{z} - z'. \nonumber
\end{align*}
To see this, note that the ideal generating the strict transform is obtained by factoring out the largest possible power of $(\hat{z})$. This power is simply $(\hat{z})$, since in the ideal above, setting $\hat{z} = 0$ produces the nontrivial ideal
$$ \langle \hat{u} - 1 - x'_1 \hat{v}, \quad y', \quad -z' \rangle. $$

Step 5: Computing images of $E$. The strict transform $s_{\str, 1}$ of $s$ to the affine space $\hat{A} \times B$ is a $(d - 1)$-to-$(d-1)$ rational correspondence, and its graph is contained in 
$$V_{\str, 1} \cap (C_{\str, 1} \times D_{\str, 1} \times C \times D).$$
The difference between these sets is that, in the definition of the secant correspondence, we removed the ``identity component'', see Definition \ref{def_secant}.

By the result of Step 3, a point $e \in E \cap \hat{A}$ is determined by a choice of $\hat{v}$. Suppose $\hat{v} \neq 0$. Setting $\hat{u}, \hat{y}, \hat{z} = 0$ in the ideal of $V_{\str, 1}$ gives the equations 
$$-1 - \hat{v} x'_1 = 0, \quad -y' = 0, z' = 0.$$
Intersecting with $C \times D$ in the second factor, the images of $s_{\str, 1}$ in $\hat{B}$ include all the points of the form $(x'_0, x'_1, 0, 0)$, where $(x'_0, x'_1)$ is an intersection of the ``horizontal'' line $x'_1 = -1/\hat{v}$ with $C$. This is because none of these points correspond to pairs in the identity component, since they have different projections to $C$. This provides $d - 1$ images, by Bezout's theorem and the assumption that the scratch point $p$ is basic. Since $s_{\str}$ was $(d-1)$-to-$(d-1)$, these are all the images. Further, we have shown that $s_{\str}$ generically identifies $E$ with the pencil of lines through $c$ in $X = \PP^2$. This proves (1).

Step 6: The remaining points of $E$. We have thus far described the images of any point in $E \cap \hat{A}$ such that $\hat{v} \neq 0$, but there remain two points of $E$. Since the rational correspondence is a projective variety, we may take a flat limit to find that the point where $\hat{v} = 0$ corresponds to the line at infinity in $X$, and the point where $\hat{v} = \infty$ corresponds to the tangent line to $c$.

Step 7: From set-theoretic images to pushforwards. By the results of Step 5 and Step 6, we have
$$ s_{\str}(E) = (C \times q)_{\str}.$$
Therefore, to prove the claim about the pushforward, we need only check the multiplicity. A point in $(C \times q)_{\str}$ has $d-2$ images with multiplicity in the main affine patch, so the image in $E$ appears with multiplicity $1$. Then use self-adjointness.
\end{proof}

\begin{proof}[Proof of Lemma \ref{lemma_reflect_isotropic}]
    We study reflection at the blowup of a basic isotropic scratch point.
    
    Step 1: As in Lemma \ref{lemma_secant_at_isotropic}, we may reduce to the case that $c = [0 : 0 : 1]$ with tangent slope $i$.
    
    Step 2: Define affine charts
    $$A = \Spec k[x_0, x_1, w_1, w_2],$$
    $$B = \Spec k[x'_0, x'_1, q'_0, q'_1],$$
    where
    $$x_0 = X_0/X_2, \quad x_1 = X_1/X_2, \quad w_1 = Q_1 / Q_0 - i, \quad w_2 = Q_2 / Q_0,$$
    $$x'_0 = X'_0 / X'_2, \quad x'_1 = X'_1 / X'_2, \quad q'_0 = Q_0 / Q_2, \quad q'_1 = Q_1 / Q_2.$$
    Define 
    $$\hat{A} = \Spec k[\hat{x}_0, \hat{x}_1, \hat{w}_1, \hat{w}_2],$$
    where
    $$\hat{x}_0 \hat{w}_2 = x_0, \quad \hat{x}_1 \hat{w}_2 = x_1, \quad \hat{w}_1 \hat{w}_2 = w_1, \quad \hat{w}_2 = w_2.$$
    Let the affine equation of $C$ be $F(x_0, x_1) = 0$.
    
    Over $A \times B$, we may define $\Gamma_r$ directly from its ideal without taking a Zariski closure, since the condition $q \neq q'$ holds on $A \times B$. Therefore we have 
    $$\Gamma_r|_{A \times B} = (C \times D)^2 \cap V,$$
    where $V$ is the variety defined by
    \begin{align}
        &x_0 = x'_0, \nonumber \\
        &x_1 = x'_1, \nonumber \\
        &\begin{vmatrix}
            1 & q_0' & \partial F / \partial x_0 \\
            w_1 + i & q_1' & \partial F / \partial x_1 \\
            w_2 & 1 & 0
        \end{vmatrix} = 0. \label{eq_r_iso_det}
    \end{align}
    
    Step 3: As in Lemma \ref{lemma_secant_at_infty}, the strict transforms of $C, D$ in $\hat{A}$ are
    \begin{align*}
        C_{\str, 1}: & \quad 0 = i \hat{x}_0 - \hat{x}_1 + O(\hat{w}_2),\\
        D_{\str, 1} &: \quad 0 = 2i \hat{w}_1 + O(\hat{w}_2).
    \end{align*}
    
    The exceptional divisor $E$ is defined by
    $$i \hat{x}_0 = \hat{x}_1, \quad \hat{w}_1 = 0, \quad \hat{w}_2 = 0.$$

    Step 4: The strict transform of $r$. First, note that
    \begin{align*}
        \frac{\partial F}{\partial x_0} &= i + 2 \kappa x_0 + O(x_0^2, ix_0 - x_1),\\
        \frac{\partial F}{\partial x_1} &= -1 + O(ix_0 - x_1).
    \end{align*}
    Let $V_{\tot, 1}$ denote the total transform of $V$ in $\hat{A} \times B$; its equations are 
    \begin{align}
        &\hat{x}_0 \hat{w}_2 = x_0', \nonumber\\
        &\hat{x}_1 \hat{w}_2 = x_1', \nonumber\\
    &\begin{vmatrix}
        1 & q'_0 & i + 2 \kappa \hat{x}_0 \hat{w}_2 + O(\hat{x}_0^2 \hat{w}_2^2, (i \hat{x}_0 - \hat{x}_1) \hat{w}_2)  \\
        i + \hat{w}_1 \hat{w}_2 & q'_1 & -1 + O((i \hat{x}_0 - \hat{x}_1) \hat{w}_2) \\
        \hat{w}_2 & 1 & 0 \\
    \end{vmatrix} = 0. \label{eq_r_iso_det_tot}
    \end{align}
    Let us compute ${\Gamma_r}_{\str, 1}$. The local equation for $C_{\str, 1}$ in $\hat{A}$ allows us to rewrite \eqref{eq_r_iso_det_tot} as
    \begin{align*}
    0 &= \begin{vmatrix}
        1 & q'_0 & i + 2 \kappa \hat{x}_0 \hat{w}_2 + O(\hat{w}_2^2)  \\
        i + \hat{w}_1 \hat{w}_2 & q'_1 & -1 + O(\hat{w}_2^2) \\
        \hat{w}_2 & 1 & 0 \\
    \end{vmatrix} \\
    &=
    \begin{vmatrix}
        2 \kappa i \hat{x}_0 \hat{w}_2 + O(\hat{w}_2^2) & q'_0 & i + 2 \kappa \hat{x}_0 \hat{w}_2 + O(\hat{w}_2^2)  \\
        \hat{w}_1 \hat{w}_2 + O(\hat{w}_2^2) & q'_1 & -1 + O(\hat{w}_2^2) \\
        \hat{w}_2 & 1 & 0 \\
    \end{vmatrix} & \text{(by column operations).}
    \end{align*}
    We may thus factor out exactly one factor of $\hat{w}_2$, from the first column, so the strict transform is
        \[
    \begin{vmatrix}
        2 \kappa i \hat{x}_0 + O(\hat{w}_2) & q'_0 & i + O(\hat{w}_2)  \\
        \hat{w}_1 + O(\hat{w}_2) & q'_1 & -1 + O(\hat{w}_2) \\
        1 & 1 & 0 \\
    \end{vmatrix} = 0.
    \]
    Then imitate Steps 5, 6, and 7 of Lemma \ref{lemma_secant_at_infty}.
\end{proof}

\section{The Improved Dynamical Degree Bound} \label{sect_global_bound}
In Section \ref{sect_cheap}, we proved a ``cheap'' dynamical degree bound for the billiards correspondence
$$b : C \times D \vdash C \times D$$
for a general curve $C$ of degree $d \geq 2$. In this section, we see that blowing up all the indeterminacy points of $b$ buys us a better bound. We prove our results first for general curves in Theorem \ref{thm_expensive_upper_bound}; then we pass to the space of all smooth curves. This proves our main result, Theorem \ref{thm_main}. As a bonus, we get a new proof of the complete integrability of billiards in an ellipse, also known as Poncelet's Porism (Corollary \ref{cor_poncelet}).

From now on, fix a general degree $d \geq 2$ curve $C$ in $\PP^2$. The only generality assumptions we require are the following.
\begin{enumerate} [label=(G\arabic*)]
    \item The curve $C$ is smooth (hence irreducible, reduced, and contains no lines). \label{GA_smooth}
    \item There are $d$ distinct points in $C \cap L_\infty$. \label{GA_simple_inf}
    \item The tangent slope at each point of $C_\infty$ is non-isotropic. \label{GA_no_inf_isotropic}
    \item Every isotropic tangent of $C$ is only tangent to $C$ at a single point, and this tangency is simple. \label{GA_simple_isot}
\end{enumerate}

Each condition above is open and holds for at least one degree $d$ curve, so they all hold generically.

\begin{remark}
Artificial as \ref{GA_simple_isot} appears, it is of fundamental importance in algebraic billiards. Indeed, among pairs of real conics, the confocal conics are exactly those that share their isotropic tangents. Confocal conics produce integrable billiards, so it is reasonable to expect \ref{GA_simple_isot} to influence the dynamics (see also \cite{MR3236494, MR4238126}).
\end{remark}

We now define the birational model used for the improved bound of Theorem \ref{thm_main}.

We give names to the scratch points of $C \times D$:
$$p^{(\infty)}_1, \ldots, p^{(\infty)}_{2d}, \quad p^{(i)}_1, \ldots, p^{(i)}_{d(d-1)}, \quad p^{(-i)}_1, \ldots, p^{(-i)}_{d(d-1)}. $$

The points $p^{(\infty)}_j$ are the scratch points at infinity, that is, the indeterminacy points of $s: C \times D \vdash C \times D$. They are the points of $C \times D$ of the form $(c, q)$, where $\pi(c) \in C_\infty$ and $[q] = t(c)$. By \ref{GA_simple_inf} and \ref{GA_no_inf_isotropic}, there are $2d$ of them. 

The points $p^{(\pm i)}_j$ are the isotropic scratch points, that is, the indeterminacy points of $r: C \times D \vdash C \times D$. These are exactly the points $(c, [1: \pm i:0]) \in C \times D$ such that $c$ is an isotropic tangency point of $C$ of tangent slope $\pm i$. There are $d(d-1)$ of each kind, by \ref{GA_no_inf_isotropic}, \ref{GA_simple_isot}, and Lemma \ref{lemma_rh}.

\begin{defn} \label{def_mod_phase_space}
The \emph{modified billiards phase space} $P$ is the projective surface obtained from $C \times D$ by blowing up each of the $2d^2$ points
\[
p_1^{(\infty)}, \ldots p_{2d}^{(\infty)}, p^{(i)}_1, \ldots, p^{(i)}_{d(d-1)}, p^{(-i)}_1, \ldots, p^{(-i)}_{d(d-1)}.
\]
\end{defn}

We work throughout in a preferred basis for $\Num^1 P$, defined as follows.
Choose general points $c_0 \in C$ and $q_0 \in D$. Let
$$C^{(0)} = C \times q_0, \quad D^{(0)} = c_0 \times D.$$
For each $j$, let $E^{(\infty)}_j \subset P$ be the exceptional divisor above $p^{(\infty)}_j$. Similarly define $E^{(i)}_j$, $E^{(-i)}_j \subset P$.
Our basis shall be
\begin{equation} \label{eq_basis}
C^{(0)}, \quad D^{(0)}, \quad E_1^{(\infty)}, \ldots, E_{2d}^{(\infty)}, \quad E_1^{(i)}, \ldots, E_{d(d-1)}^{(i)}, \quad E_1^{(-i)}, \ldots, E_{d(d-1)}^{(-i)}.
\end{equation}
For convenience in indexing, we also sometimes denote these divisors by
$$\Delta_1, \ldots, \Delta_{2d^2 +2}.$$

\begin{lemma}
We have $\Num^1 P \cong \Z^{2 + 2d^2}$, via the choice of basis $\Delta_1, \ldots, \Delta_{2d^2 + 2}$. The intersection form on $\Num^1 P$ in our preferred basis is
\[
\small
J_P :=
\left(
\begin{array}{cc|c|c}
0    &  1     &  &   \\
1     & 0     &  &   \\ \hline
      &       &-1 &  \\ \hline
      &       &   &-1
\end{array}
\right)
\]
where the blocks are of size $2$, $2d$, and $2d(d-1)$.
\end{lemma}

\begin{proof}
We see this by starting with the intersection form on $\Num^1 C \times D$ with basis $C^{(0)}$ and $D^{(0)}$, which is given by
\[
\small
\begin{pmatrix}
    0 & 1 \\
    1 & 0
\end{pmatrix}.
\]
Exceptional divisors of blowups are $(-1)$-curves; see Hartshorne \cite[Chapter V]{Hartshorne}. Since $C^{(0)}$ and $D^{(0)}$, viewed in $C \times D$, do not pass through any blown-up points, we obtain $J_P$.
\end{proof}

\begin{prop} \label{prop_strict_s_matrix}
Let $C$ be a general degree $d \geq 2$ curve in $\PP^2$. Let $s: C \times D \vdash C \times D$ be the secant correspondence, and let $\hat{s}: P \vdash P$ be the strict transform of $s$ to the modified phase space. The matrix of $\hat{s}_* : \Num^1 P \to \Num^1 P$ is
\[
M_{\hat{s}} \colonequals 
\small
\left(
\begin{array}{cc|ccc|ccc}
d - 1 &  2     & 1   & \hdots & 1 &   & &   \\
0     &  d - 1 & 0   & \hdots & 0 &   & &\\ \hline
0     &  -1    & -1 & 0      & 0  &   & & \\
\vdots& \vdots & 0   & \ddots & 0 &   & & \\
0     & -1     & 0   & 0      & -1 &    &  &\\ \hline
      &        &     &        &    & 1 & 0 & 0 \\
      &        &     &        &    & 0 & \ddots & 0\\
      &        &     &        &    & 0 & 0 & 1 \\
\end{array}
\right).
\]
The blocks are of size $2$, $2d$, and $2d(d-1)$ left to right and top to bottom. The blocks shown as empty are zero matrices.
\end{prop}

\begin{proof}
We proceed as in the proof of Theorem \ref{thm_cheap_upper_bound}. Let $\Pi_1^*, \Pi_2^* : P \times P \to P$ be the projections to each factor. For any divisors $\Delta, \Delta'$, let $\iota_{\hat{s}}(\Delta, \Delta') = \Gamma_{\hat{s}} \cdot \Pi_1^* \Delta \cdot \Pi_2^* \Delta'$. Let
$$M'_{\hat{s}} \colonequals \left( \iota_{\hat{s}}(\Delta_i, \Delta_j ) \right)_{1 \leq i, j \leq 2d^2 + 2}.$$
The explicit formula for $J_P$ shows that the dual basis for $\Delta_1, \Delta_2, \ldots, \Delta_{2d^2 + 2}$ is given by intersecting against $\Delta_2, \Delta_1, -\Delta_3, \ldots, -\Delta_{2d^2 + 2}$, so
\begin{equation}
\label{eq_strict_s_intersection_matrix}
    M_{\hat{s}} = J_P M'_{\hat{s}}.
\end{equation}
Thus, it suffices to compute all the values $\iota_{\hat{s}}(\Delta_i, \Delta_j )$.

For the upper-left block, the calculation \eqref{eq_cheap_s_matrix} of $M'_s$ in the proof of Theorem \ref{thm_cheap_upper_bound} applies verbatim.

We compute the middle-left and lower-left blocks indirectly. We observe that, since $\iota_{\hat{s}}$ is symmetric, we can obtain these entries from those in the first two rows of $M'_{\hat{s}}$ and then using \eqref{eq_strict_s_intersection_matrix}. Thus we need only compute the second and third blocks of columns.

We compute the second block of columns directly. For each $j$ in $1 \leq j \leq 2d$, we claim
\begin{equation} \label{eq_strict_s_pull_infty}
    \hat{s}_* E^{(\infty)}_j = C^{(0)} - E_j^{(\infty)}.
\end{equation}
Indeed, say $p_j^{(\infty)} = (c_j, q_j)$. Let $P_j \to C \times D$ be the blowup of $C \times D$ at $p_j^{(\infty)}$. Let $\hat{s}_j$ be the strict transform of $s$ to $P_j$. Let $C_j$ denote the strict transform of $C \times q_j$ to $P_j$. Applying Lemma \ref{lemma_secant_at_infty}, on $P_j$, we have 
$$\hat{s_j}_* E^{(\infty)}_j = C_j$$
as divisors and
$$\hat{s_j} (E^{(\infty)}_j) = C_j$$
as sets.
By \ref{GA_simple_inf} and \ref{GA_no_inf_isotropic}, the curve $C_j$ does not pass through any other blown-up points of $P \to P_j$, so the blowup is an isomorphism above an open subset of $P_j$ containing $C_j$. Identifying $C_j$ with its pullback to $P$, we have
$$\hat{s}_* E^{(\infty)}_j = C_j.$$
Let $\equiv$ denote numerical equivalence. Since $q_j \equiv q_0$ on $D$, we have
$$C_j \equiv C \times q_0 = C^{(0)}$$
on $C \times D$. The divisor $C^{(0)}$, viewed on $C \times D$, passes through $p_j^{(\infty)}$ with multiplicity $1$, so on $P$, we have
$$C_j \equiv C^{(0)} - E_j^{(\infty)}.$$
This proves \eqref{eq_strict_s_pull_infty}.

Now we compute the third group of columns directly, by a similar argument. For each $j$ in $1 \leq j \leq d(d-1)$, we claim
\begin{equation}
    \label{eq_strict_s_pull_isot}
    \hat{s}_* E_j^{(i)} = E_j^{(i)}.
\end{equation}
The argument for the $E_j^{(-i)}$ is identical, so proving \eqref{eq_strict_s_pull_isot} suffices to compute the third block of columns. Say $p_j^{(i)} = (c_j, i)$. Let $P_j \to C \times D$ be the blowup of $C \times D$ at $p_j^{(i)}$. Let $\hat{s}_j$ be the strict transform of $s$ to $P_j$. Applying Lemma \ref{lemma_secant_at_isotropic}, we have
$$\hat{s_j}_* E^{(i)}_j = E^{(i)}_j$$
as divisors. Also by Lemma \ref{lemma_secant_at_isotropic}, if $f_j, g_j: \Gamma_{\hat{s}_j} \to P_j$ are the coordinate projections, then $f_j^{-1}(E_j^{(i)})$ is the union of a curve $E$ such that $g_j(E) = E_j^{(i)}$ and some points with $D$-coordinate $i$. Then, by \ref{GA_no_inf_isotropic} and \ref{GA_simple_isot}, the birational morphism $P \to P_j$ induces an isomorphism on a Zariski open set containing $\hat{s}_j(E_j^{(i)})$. This proves \eqref{eq_strict_s_pull_isot}.
\end{proof}

\begin{prop} \label{prop_strict_r_matrix}
Let $C$ be a general degree $d \geq 2$ curve in $\PP^2$. Let $r: C \times D \vdash C \times D$ be the reflection correspondence. Let $\hat{r}: P \vdash P$ be the strict transform of $r$ to the modified phase space. Then $\hat{r}$ is an automorphism of $P$, and the matrix of $\hat{r}_* : \Num^1 P \to \Num^1 P$ is
\[
M_r \colonequals
\small
\left(
\begin{array}{cc|ccc|ccc}
1     &  0     &     &       & & 0 & \hdots & 0  \\
d(d-1)&  1     &     &       & & 1 & \hdots & 1  \\ \hline
      &        &  1 & 0      & 0  &   & & \\
      &        & 0   & \ddots & 0 &   & & \\
      &        & 0   & 0      & 1 &  &  &\\ \hline
-1    & 0      &     &        &    & -1 & 0 & 0 \\
\vdots&\vdots  &     &        &    & 0 & \ddots & 0\\
-1    & 0      &     &        &    & 0 & 0 & -1 \\
\end{array}
\right)
\]
The blocks are of size $2$, $2d$, and $2d(d-1)$ left to right and top to bottom. The blocks shown as empty are zero matrices.
\end{prop}

\begin{proof}
    The proof structure closely follows that of Proposition \ref{prop_strict_s_matrix}. For the upper-left block, the computation \eqref{eq_cheap_r_matrix} of $M'_r$ in the proof of Theorem \ref{thm_cheap_upper_bound} applies verbatim. The remaining entries in the first block of columns can be obtained from the second and third blocks of columns by symmetry of $\iota_{\hat{r}}$. Thus it suffices to directly compute the second and third blocks of columns.

    For the second block of columns, for each $j$ in $1 \leq j \leq 2d$, we claim
    \begin{equation}
    \label{eq_strict_r_pull_infty}
    \hat{r}_* E^{(\infty)}_j = E^{(\infty)}_j.
    \end{equation}
    Let $P_j$ be the blowup of $C \times D$ at $p_j^{(\infty)}$. By Lemma \ref{lemma_reflect_infty}, we have
    $$\hat{r_j}_* E_j^{(\infty)} = E_j^{(\infty)}$$
    as divisors and
    $$\hat{r_j}( E_j^{(\infty)}) = E_j^{(\infty)}$$
    as sets. The birational morphism $P \to P_j$ induces an isomorphism on a Zariski open set containing $\hat{r_j}( E_j^{(\infty)})$. This proves \eqref{eq_strict_r_pull_infty}.

    For the third block of columns, for each $j$ in $1 \leq j \leq d(d-1)$, we claim
    \begin{equation}
        \label{eq_strict_r_pull_isot}
        \hat{r}_* E_j^{(i)} \equiv D^{(0)} - E_j^{(i)}.
    \end{equation}
    Indeed, say $p^{(i)}_j = (c_j, i)$. Let $P_j \to C \times D$ be the blowup of $C \times D$ at $p_j^{(i)}$. Let $\hat{r}_j$ be the strict transform of $r$ to $P_j$. Let $D_j$ denote the strict transform of $c_j \times D$ to $P_j$. Applying Lemma \ref{lemma_reflect_isotropic}, on $P_j$, we have
    $$ \hat{r_j}_* E_j^{(i)} = D_j$$
    as divisors and
    $$ {\hat{r}_j}(E_j^{(i)}) = D_j$$
    as sets. The birational morphism $P \to P_j$ is an isomorphism above an open subset of $P_j$ containing $D_j$. Identifying $D_j$ with its pullback to $P$, we have
    $$\hat{r}_* E_j^{(i)} = D_j.$$
    Since $c_j$ is numerically equivalent to $c_0$ on $C$, we have that $c_j \times D$ is numerically equivalent to $D^{(0)}$ on $C \times D$. And $c_j \times D$ intersects $p_j^{(i)}$ with multiplicity $1$, so pulling back to $P$, we get a numerical equivalence
    $$ D_j \equiv D^{(0)} - E_j^{(i)}. $$
    This proves \eqref{eq_strict_r_pull_isot}. The case for $E_j^{(-i)}$ follows by symmetry.

    The fact that $\hat{r}$ is an automorphism is a consequence of the case-by-case analysis above. We know that $\hat{r}$ is a birational involution, so it suffices to show that it has no contracted curves. If a curve is contracted by $\hat{r}$, it is either the strict transform of a curve contracted by $r$, or it is an exceptional divisor of $P \to C \times D$. But the contracted curves of $r$ are mapped onto exceptional divisors by $\hat{r}$, and we showed that no exceptional divisor is contracted by $\hat{r}$.
\end{proof}

\begin{prop} \label{prop_strict_b_matrix}
    Let $C$ be a general degree $d \geq 2$ curve in $\PP^2$. Let $b = r \circ s$ be the billiards correspondence on $C \times D$. Let $\hat{b} : P \vdash P$ be the strict transform of $b$ to the modified phase space. Then the matrix of $\hat{b}_*: \Num^1 P \to \Num^1 P$ is
    \begin{equation} \label{eq_prod_matrix}
M_{\hat{b}} = M_{\hat{r}} M_{\hat{s}} = \quad \small \left(
\begin{array}{cc|ccc|ccc}
d-1     &  2              &  1     & \hdots     &1 & 0 & \hdots & 0  \\
d(d-1)^2& (2d+1)(d-1)     &  d(d-1)&  \hdots    &d(d-1) & 1 & \hdots & 1  \\ \hline
 0     &  -1      & -1 & 0      & 0  &   & & \\
\vdots & \vdots &  0   & \ddots & 0 &   & & \\
 0     & -1      & 0   & 0      & -1 &  &  &\\ \hline
-(d-1)& -2      &-1   & \hdots & -1 & -1 & 0 & 0 \\
\vdots&\vdots  &\vdots& \ddots& \vdots& 0 & \ddots & 0\\
-(d-1)& -2      &-1   &\hdots& -1   & 0 & 0 & -1 \\
\end{array}
\right).
\end{equation}
The blocks are of size $2$, $2d$, and $2d(d-1)$ left to right and top to bottom. The blocks shown as empty are zero matrices.
\end{prop}

\begin{proof}
    By our generality assumptions,  Lemma \ref{lemma_secant_at_isotropic}, and Lemma \ref{lemma_secant_at_infty}, the contractions of $\hat{s}$ are all of the following form: if $\ell$ is an isotropic tangent line of $C$ at $c$, then there is a contraction
    $$((E_j^{\pm i}, \pm i), (c', \pm i)),$$
    for each intersection point
    $$c' \in \ell \cap C \smallsetminus c.$$
    Again by generality, no such point $(c', \pm i)$ is an indeterminacy point of $\hat{r}$. So, by Theorem \ref{thm_dd_props} (\ref{it_functorial}), we have
    $$\hat{b}_* = \hat{r}_* \hat{s}_*.$$
    Then the result follows from matrix multiplication.
\end{proof}

The following lemma provides the crucial estimate of the spectral radius of $\hat{b}_*$. The leading eigenvalue turns out to be a real cubic integer. Our proof of this is ad hoc; it would be very interesting to understand the geometry behind it.

\begin{lemma} \label{lemma_main_bound}
    Let $C$ be a general curve of degree $d \geq 2$ in $\PP^2$. Then
    $$\rad \hat{b}_* = \rho_d,$$
    where $\rho_d$ is the largest root of the cubic polynomial
    $$ \Phi_d(\lambda) \colonequals \lambda^3 - (2 d^2 - d - 3) \lambda^2 + (2d^2 - 4d + 3) \lambda - (d - 1). $$
    The algebraic integer $\rho_d$ has the following properties:
    \begin{enumerate}
        \item The algebraic integer $\rho_d$ is a real number in the interval $(2d^2 - d - 5, 2d^2 - d - 3)$.
        \item When $d = 2$, we have $\rho_d = 1$. 
        \item If $d \geq 3$, then the polynomial $\Phi_d(\lambda)$ is irreducible and $\rho_d$ is a Pisot number.
        \item If $d = 3$, the polynomial $\Phi_d(\lambda)$ has one real root, equal to $\rho_d$. If $d \geq 4$, the polynomial $\Phi_d(\lambda)$ is real-rooted.
    \end{enumerate}
\end{lemma}

\begin{proof}
Let $M = M_{\hat{b}}$ be as defined in \eqref{eq_prod_matrix}. Let $\chi_M(\lambda)$ be the characterstic polynomial of $M$. We claim that
\begin{equation} \label{eq_big_char_poly}
    \chi_M(\lambda) = \Phi_d(\lambda) (\lambda + 1)^{2d^2 - 2} (\lambda - (d - 1)).
\end{equation}
To prove \eqref{eq_big_char_poly}, we replace $M$ by a conjugate with a simpler form.
Define a matrix
\[
\Psi =
\tiny
\left(
\begin{array}{cc|cccc|cccc}
1     &  0  &&&&&&&     \\
0     &  1  &&&&&&&     \\ \hline
&& 1     & 1      &  \hdots  & 1 &   \\
&&      & -1     & 0        & 0  &  \\
&&     &       & \ddots  & 0  &  \\
&&      &      &  &       -1 & &  \\ \hline
&&&&&& 1 & 1 & \hdots & 1 \\
&&&&&&   &-1 &        &   \\
&&&&&&   &   & \ddots &   \\
&&&&&&   &   &        &-1 \\
\end{array}
\right).
\]
We note that $\Psi = \Psi^{-1}$, although this is not needed in our computation.
Let $\Pi$ be the permuation matrix associated to the permutation with cycle presentation
$$ (4 \quad 5 \quad \dots \quad 2 + 2d + 1).$$

Conjugating $M$ by $\Psi$, then by $\Pi$, gives a block matrix of the form
\[
\Pi \Psi M \Psi^{-1} \Pi^{-1} = 
\left(
\begin{array}{c|c}
A & 0 \\ \hline
B & -1 \\
\end{array}
\right).
\]
Here $A$ is the $4 \times 4$ matrix
\[
\Pi \Psi M \Psi^{-1} \Pi^{-1} =
\small
\left(
\begin{array}{cccc}
d - 1      & 2             & 1        & 0  \\
d(d-1)^2   & (2d + 1)(d-1) & d(d-1)   & 1  \\
0          & -2d           & -1       & 0  \\
-2d(d-1)^2 & -4d(d-1)      & -2d(d-1) & -1 \\
\end{array}
\right),
\]
and $B$ is some matrix, the entries of which do not affect our computation. More importantly, the upper-right block is $0$, and the lower-right block is the negative of the identity matrix. It follows that the characteristic polynomial of $M$ factors:
$$\chi_M(\lambda) = (\lambda + 1)^{2d^2 - 2} \chi_A (\lambda).$$
The characteristic polynomial of $A$ can be computed directly:
$$ \chi_A(\lambda) = (\lambda - (d - 1)) \left(\lambda^3 - (2 d^2 - d - 3) \lambda^2 + (2d^2 - 4d + 3) \lambda - (d - 1)\right). $$
The cubic factor is the polynomial $\Phi_d(\lambda)$, proving \eqref{eq_big_char_poly}.

If $d = 2$, then 
$$\chi_M(\lambda) = (\lambda + 1)^6 (\lambda - 1)^4,$$
so $\rad M = 1$.

If $d \geq 3$, we claim that the spectral radius of $M$ is the maximum absolute value among roots of $\Phi_d(\lambda)$, and that this value is between $2d^2 - d - 5$ and $2d^2 - d - 3$. We check all claims for the case $d = 3$ individually using a computer algebra system.

If $d \geq 4$, then $\Phi_d(\lambda)$ has only real roots. Indeed, we can test directly that $\Phi_d(0)$ and $\Phi_d(1)$ are negative for all $d > 2$. We also may check that $\Phi_d(1/2)$ is a quadratic in $d$ that is positive for all $d \geq 4$. Thus, if $d \geq 4$, then $\Phi_d(\lambda)$ has a real root strictly between $0$ and $1/2$ and another real root strictly between $1/2$ and $1$. Finally, we know the sum of the real roots is the trace $2d^2 - d - 3$ of $\Phi_d(\lambda)$, and we found a pair of real roots with sum strictly between $0$ and $2$. The remaining root therefore is also real and is in the interval $(2d^2 - d - 5, 2d^2 - d - 3)$. Since $d \geq 4$, this exceeds the next-highest eigenvalue of $\chi_M(\lambda)$, which is $d - 1$.

It follows from this explicit description that $\rho_d$ is Pisot for all $d \geq 4$ (that is, its conjugates have absolute value less than $1$), and the case $d = 3$ can be checked directly.

To finish the proof of (3), we claim that if $d \geq 3$, the cubic $\Phi_d(\lambda)$ is irreducible. We check this by inspection for $d = 3$, then consider the case $d \geq 4$. Say $\Phi_d(\lambda)$ has three linear factors over $\Q$, hence over $\Z$; then $\rho_d$ divides the constant coefficient $d - 1$, but this contradicts our bounds on $\rho_d$. If $\Phi_d(\lambda)$ has a quadratic and linear factor over $\Q$, hence over $\Z$, then our bounds on the two roots besides $\rho_d$ imply that their product is not in $\Z$; and neither of these two roots are in $\Z$, so neither do they arise as linear factors. Thus $\Phi_d(\lambda)$ is indeed irreducible.

We thank Curt McMullen for the Pisot observation and the irreducibility observation above.
\end{proof}

\begin{remark}
We found the conjugating matrices $\Phi$ and $\Pi$ by guesswork with examples for low values of $d$ in Sage. The idea to look for a simpler conjugate of $M$ came from the fact that, for all low values of $d$ that we tried, the characteristic polynomial $\chi_M(\lambda)$ factored.
\end{remark}

We now have all the pieces for our main result.
\begin{thm}[= Theorem \ref{thm_main}] \label{thm_expensive_upper_bound}
Let $C$ be a smooth curve of degree $d \geq 2$ in $\PP^2$, and let $b: C \times D \vdash C \times D$ be the billiards correspondence. Let $\lambda_0(b)$, $\lambda_1(b)$, and $\lambda_2(b)$ be the dynamical degrees of $b$. Then
\begin{equation} \label{eq_dd_0_and_2}
    \lambda_0(b) = \lambda_2(b) = d - 1
\end{equation}
and
$$\lambda_1(b) \leq \rho_d,$$
where $\rho_d$ is the cubic integer of Lemma \ref{lemma_main_bound}.
In particular, when $d = 2$, we have $\lambda_1(b) = 1$.
\end{thm}

\begin{proof}
Since $b$ is a $(d-1)$-to-$(d-1)$ correspondence, we have \eqref{eq_dd_0_and_2}. 

Let $C$ be a curve satisfying \ref{GA_smooth}-\ref{GA_simple_isot}. Dynamical degrees are birational conjugacy invariants (Theorem \ref{thm_dd_props} (\ref{it_birat_inv})), so $\lambda_1(b) = \lambda_1(\hat{b})$. We have $\hat{b} = \hat{r} \circ \hat{s}$.
Then, by Theorem \ref{thm_dd_props} (\ref{it_dd_bound}), and Lemma \ref{lemma_main_bound}, we have
$$\lambda_1(\hat{b}) \leq \rad \hat{b}_* = \rho_d.$$
Since $\rho_2 = 1$, in degree $2$, we have $\lambda_1(b) = 1$.

It remains to extend the result to all smooth algebraic billiard curves, not just general ones. Let $\mc{C} \to \mc{B}$ be the family of smooth degree $d$ plane curves in $\PP^n$. Let $K$ be the function field of $\mc{B}$, and let $\mc{C}_\eta$ denote the generic geometric fiber of the family of all plane curves of degree $d$. Then $\mc{C}_\eta$, viewed as a curve over $K$, is a degree $d$ curve satisfying \ref{GA_smooth}-\ref{GA_simple_isot}. Similarly, let $D_\eta$ denote $D$ defined over the base field $K$. Let $b_\eta : C_\eta \times D_\eta \vdash C_\eta \times D_\eta$ be the billiards correspondence on $C_\eta \times D_\eta$, over $K$. Then our results so far show that $\lambda_1 (b_\eta) \leq \rho_d$.

Next, we invoke Theorem \ref{thm_dd_props} (\ref{it_specialization}) that ``in a smooth family of dynamical systems, dynamical degrees cannot increase with specialization''. The family $\mc{C} \to \mc{B}$ is projective, smooth, and surjective of relative dimension 1. Thus $\mc{C} \times D \to \mc{B}$ is projective, smooth, and surjective of relative dimension $2$. For any $p \in \mc{B}$, the induced correspondence $b_p : \mc{C}_p \times D \vdash \mc{C}_p \times D$ on the special fiber is simply the billiards correspondence on $\mc{C}_p \times D$. Thus, by Theorem \ref{thm_dd_props} (\ref{it_specialization}) applied to the restriction of $\mc{C} \times D \to \PP^N_k$ to $\mc{B}$ in the base, we have
$$\lambda_1 (b_p) \leq \lambda_1(b_\eta).$$
\end{proof}

\begin{remark} \label{rem_algstab}
If $C$ fails the generality criteria \ref{GA_smooth}-\ref{GA_simple_isot}, many of the specific intersection numbers we calculate will change, and it seems likely that $\lambda_1(b) < \rho_d$. For instance, it seems likely that if $C$ is a union of confocal conics, the dynamical degree should be quite small.

On the other hand, Conjecture \ref{conj_main} speculates that if $C$ is very general, then $P$ is algebraically stable, so $\lambda_1(b) = \rho_d$. The heuristic that justifies this conjecture is as follows. The forward orbits of contracted curves on $P$ start with direction $\pm i$. Therefore, by Lemma \ref{prop_basic_r}(\ref{prop_basic_r_isot}), these orbits have isotropic directions for all time, switching back and forth between $\pm i$. This reduces the study of destabilizing orbits on $P$ to the restriction of $b$ to $C \times D_\infty$. Roughly speaking, a destabilizing orbit of $b$ would start and end at critical points of $b|_{C \times D_\infty}$. Specific choices of the curve $C$ will have these critical orbits, but the generic curve presumably does not. However, the exponential growth in orbit size makes understanding critical orbits of correspondences rather more difficult than in the case of maps.
\end{remark}

\subsection{Conics}
The billiards map in an ellipse is famously a completely integrable system; see \cite{MR2168892}. The classical Poncelet porism is a consequence. Specializing our arguments to degree $d = 2$ and appealing to a classification theorem, we obtain a new proof of complete integrability of billiards in quadrics that works over any algebraically closed field of characteristic not equal to $2$. For other algebraic approaches to elliptic billiards and Poncelet's porism, see \cite[Chapter 11.2]{MR2683025} and \cite{MR4238126, MR497281}.

\begin{cor} \label{cor_poncelet}
Let $C$ be an irreducible and reduced quadric in $\PP^2_k$ that does not contain $[1: i : 0]$ or $[1 : -i : 0]$. Then the modified phase space $P$ is a rational elliptic surface, and $\hat{b}$ is a regular automorphism of $P$ preserving the genus $1$ fibration.
\end{cor}

\begin{proof}
    It is clear that $P$ is a rational surface, since $P$ is birational to $\PP^1 \times \PP^1$. It is the blowup of $\PP^1 \times \PP^1$ at $8$ points, so it is a rational elliptic surface. Thus it has a unique elliptic fibration.
    
    The correspondences $r$ and $s$ are each birational maps, since the degree of $C$ is $2$. The assumptions on $C$ make it general in the sense of \ref{GA_smooth}-\ref{GA_simple_isot}, so the results of this section apply. In Proposition \ref{prop_strict_r_matrix}, we proved that $\hat{r}$ is an automorphism. And $\hat{s}$ is also an automorphism when $d = 2$, by Lemmas \ref{lemma_secant_at_isotropic} and \ref{lemma_secant_at_infty}. Thus $\hat{b} = \hat{r} \circ \hat{s}$ is an automorphism. 
    The Jordan normal form of $M_{\hat{b}}$, calculated in \eqref{eq_prod_matrix}, has one block of size $3$ and seven blocks of size $1$. Since action on cohomology is functorial for morphisms, the growth of $\rad \hat{b}^m$ is quadratic and not linear. The growth classification theorem of automorphisms of rational surfaces then implies that $\hat{b}$ preserves the elliptic fibration \cite[Theorem 4.6]{MR3821147}.
\end{proof}

\begin{remark}
    The invariant fibration is also a geometric consequence of the existence of two invariant $2$-forms. The first is $\omega$ (Theorem \ref{thm_invariant_form}). The second is constructed as follows. The canonical class on $C \times D$ is of type $(-2,-2)$. The curve $\Delta = -(C \times D_\infty \cup C_\infty \times D)$ is of this type, so there is a form $\eta$ with $\dv \eta = \Delta$. The divisor of $\pi^*\eta$ is (a local calculation shows) the strict transform of $\Delta$. Then, since $\hat{b}$ is regular, the same is true of $\dv \hat{b}^* \pi^* \eta$, so we have $\pi^* \eta = c_0 \hat{b}^* \pi^* \eta$ for some $c_0 \in k^*$. Imitating \cite[Section 7.1]{cantat2023dynamics} shows $c_0 = 1$, so $\omega/\eta$ is invariant. Then the divisor calculation of Proposition \ref{prop_div} shows $\omega/\eta$ is nonconstant. 
\end{remark}

\subsection{Real applications}

We end with the proofs of Corollary \ref{cor_top} on topological entropy and Corollary \ref{cor_orbit_growth} on orbit growth.

Making sense of these topological entropy in this context is quite delicate, because the domains of definition are not compact. Here, by $h_{\Top}(b_\Omega)$, we mean the Bowen-Dinaburg definition, based on separated sets. By $h_{\Top}(b)$, we mean Friedland entropy, which extends that of Bowen-Dinaburg. For details on the relationships between these entropies, see \cite{MR2410955}.

\begin{proof}[Proof of Corollary \ref{cor_top}]
Suppose that $\Omega$ is a region such that the Zariski closure of $\partial \Omega$ in $\PP^2_\C$ is a smooth algebraic curve $C_0$ of degree $d \geq 2$. Then the graph of the map $b_\Omega$ is a subset of the graph of the correspondence $b$, by Proposition \ref{prop_b_basic}. Thus
\begin{equation} \label{eq_top_real_complex}
    h_{\Top}(b_\Omega) \leq h_{\Top}(b). 
\end{equation}
Dinh-Sibony, building on work of Gromov, proved a fundamental inequality relating dynamical degrees and topological entropy $h_{\Top}$ \cite{MR2391122, MR2026895}. We just need the special case of their result for surfaces, as follows. If $g: X \vdash X$ is a rational $m$-to-$n$ correspondence of a \emph{irreducible} complex surface, then
\begin{equation} \label{eq_NS}
    h_{\Top} (g) \leq \log \max \{n, \lambda_1(g), m\}. 
\end{equation}
Then Theorem \ref{thm_main} gives
$$h_{\Top} (b_{\Omega}) \leq h_{\Top}(b) \leq \log \max \{d -1, \lambda_1(b) \} \leq \log \rho_d.$$
\end{proof}

\begin{remark}
In degree $2$, the above proof shows that $h_{\Top}(b) = 0$. A more refined invariant is then the (real) polynomial entropy, which is $1$ for circular tables and $2$ for noncircular elliptical tables \cite{MR3754521}.
\end{remark}

Now we prove Corollary \ref{cor_orbit_growth}.

\begin{proof}[Proof of Corollary \ref{cor_orbit_growth}]
Recall that $\Omega \subset \R^2$ is a region with boundary $\partial \Omega$ contained in a smooth algebraic curve $C$ of degree $d \geq 2$. We may take $T = \partial \Omega$ in Proposition \ref{prop_b_basic} (\ref{it_real_inside_complex}), so the graph $\Gamma_{b_\Omega}$ of the real billiard map inside $\Omega$ is a subset of $\Gamma_b$ on $C \times D$. If $b^m(c \times D)$ and $c' \times D$ do not share a component, then
$$a_m (c, c') \leq b^m(c \times D) \cdot (c' \times D).$$
By numerical equivalence, we have
\begin{align*}
    b^m(c \times D) \cdot (c' \times D) &= b^m (c \times D) \cdot (c \times D) \\
    &= \deg_{c \times D} (b^m).
\end{align*}
Then applying Theorem \ref{thm_main} and Theorem \ref{thm_dd_props} (\ref{it_not_ample}) gives, for any $\epsilon > 0$, that
$$\deg_{c \times D} (b^m) = O( (\lambda_1(b) + \epsilon)^m ) \leq O((\rho_d + \epsilon)^m).$$
\end{proof}

\section{Singularity Confinement} \label{sect_sc}
In this section, we prove Theorem \ref{thm_sing_conf}. We show that, for a general degree $d \geq 2$ curve $C$, the billiards correspondence $b: C \times D \vdash C \times D$ satisfies a singularity confinement property. Roughly speaking, we show that the contracted curves of $b$ are not contracted by $b^2$. In slogan form, the singularity confinement property is that ``what blows down must blow up.''  The word ``singularity'' here means a mapping singularity, that is, a contraction. ``Confinement'' refers to the finite duration of time that the contraction exists. This formalizes the idea behind Figure \ref{fig_sc}.

In the case of correspondences, there are some subtleties that prevent us from applying this heuristic definition. The image of a curve by a correspondence may have several components that we wish to deal with separately: some curves, some points, even a curve with a point inside it that should be counted separately. Thus it is most convenient to work with contractions and expansions rather than contracted curves and indeterminacy points.

Recall that a \emph{contraction} of a correspondence $f: X \vdash X$ is an irreducible curve $V \subset \Gamma_f$ such that the projection of $V$ to the second factor is a point (Definition \ref{def_contraction}). Recall as well that, for any composition $g \circ f$ of correspondences, there are natural rational maps $\Gamma_{g \circ f} \dashrightarrow \Gamma_f$ and $\Gamma_{g \circ f} \dashrightarrow \Gamma_g$. See Remark \ref{rem_composite_nat_arrows}.

\begin{defn}[Singularity confinement for surface correspondences] \label{def_sing_conf}
    Let $X$ be a surface. Let $f: X \vdash X$. For each iterate $f^m: X \vdash X$, $m \geq 2$, there is a commutative diagram
    \begin{center}
\begin{tikzcd}[row sep=small, column sep=small]
     &        & \Gamma_{f^m} \arrow[dl, dashed, "\alpha"] \arrow[ddll,bend right=40] \arrow[dr, dashed] \arrow[ddrr, bend left=40, "\beta"] &   & \\
     & \Gamma_f \arrow[dl] \arrow[dr] & & \Gamma_{f^{m-1}} \arrow[dl] \arrow[dr] & \\
X &    & X &     & X
\end{tikzcd}
\end{center}.

    Let $V \subset \Gamma_f$ be a contraction. We say $(f, V)$ has the \emph{singularity confinement property} if there exists an iterate $f^m$ such that, in the above diagram, the set $\beta(\alpha^{-1}(V \smallsetminus \Ind \alpha))$ is a curve. We say $f$ itself has the singularity confinement property if, for every contraction $V$, the pair $(f, V)$ has the singularity confinement property.
\end{defn}

\begin{lemma} \label{lemma_sc_crit}
    Suppose that $f:X \vdash X$ is a surface correspondence. The following criterion is sufficient for $f$ to have the singularity confinement property. There exists a birational morphism $\pi: P \to X$ such that: for each contraction $V = (V_0 \times p) \subset \Gamma_f$, there exist $m \geq 0$ and irreducible curves $V'_0, V'_1, \ldots, V'_m \subset P$ satisfying:
    \begin{itemize}
        \item $V'_0$ is the strict transform of $V_0$;
        \item The strict transform $\hat{f}:P\vdash P$ has $V'_i \cap \Ind \hat{f} = \emptyset$;
        \item For each $1 \leq i \leq m$, we have $\hat{f}(V_i) \supset V_{i+1}$;
        \item $\pi(V'_1) = p$;
        \item $\pi(V'_m)$ is $1$-dimensional.
    \end{itemize}
\end{lemma}

\begin{proof}
The strict transform of a curve $V \subset X$ by a rational correspondence $\pi:Y \vdash X$ is the Zariski closure of $\pi^{-1}(V \smallsetminus \Exc \pi)$ in $Y$. The strict transform can be computed after restricting $V$ to a Zariski dense subset, and where a map is regular, strict transform agrees with image and respects composition.

A birational morphism $\pi: P \to X$ induces a birational morphism $\Gamma_{\hat{f}} \to \Gamma_f$. The strict transform of $V_0 \times p$ to $\Gamma_{\hat{f}}$ contains $V'_0 \times E$, and by induction, the strict transform of $V_0 \times p$ to $\Gamma_{\hat{f}^m}$ contains $V'_0 \times V'_m$. This gives a curve $V_m = \pi(V'_m)$ in the strict transform of $V_0 \times p$ by $\beta \circ \alpha^{-1}$. 
\end{proof}

Before proceeding, we pause to note that the singularity confinement property is distinct from the phenomenon of algebraic instability, though the two concepts are closely related.

\begin{example}[Confined and unstable]
The map $f: \A^2 \dashrightarrow \A^2$ defined by $f(x, y) = (x, x/y)$ has singularity confinement. The only contracted curve of $f$ is the locus $V$ cut out by $x = 0$. The indeterminacy locus $\Ind f$ is the locus of points where $y = 0$, and $f(V \smallsetminus \Ind f) = \{(0,0)\}$. But $f^2$ is the identity, so $f^2(V \smallsetminus \Ind f^2) = V$.
\end{example}

\begin{example}[Unconfined and unstable] \label{ex_unstable_and_unconfined}
Let $k = \C$ and let $f(x, y) = (x^3, x/y^2)$. Let $V$ be the locus where $x = 0$. A surface map $f$ is algebraically unstable if and only if a contracted curve $V$ is eventually mapped into $\Ind f$ \cite[Theorem 1.14]{MR1867314}.  Then $f(V \smallsetminus \Ind f) = (0,0) \in \Ind f$, so $V$ witnesses the failure of algebraic stability for $f$. Yet $f^2(x, y) = (x^9, x y^4)$, so $f^2$ also contracts $V$ to a point. Thus we cannot detect singularity confinement for $(f, V)$ by just looking at the destabilizing orbit.
\end{example}

The main result of this section is as follows.

\begin{thm}[= Theorem \ref{thm_main_sc}] \label{thm_sing_conf}
    Let $C$ be a general plane curve of degree $d \geq 2$. Then the billiards correspondence $b: C \times D \vdash C \times D$ has the singularity confinement property.
\end{thm}

\begin{proof}
Suppose that $V \subset \Gamma_b$ is a contraction. Writing $b = r \circ s$, we see that the image of $V$ in $\Gamma_s$ or $\Gamma_r$ must be a contraction. Proposition \ref{prop_basic_s} and Proposition \ref{prop_basic_r} contain classifications of the contractions of $\Gamma_s$ and $\Gamma_r$. 

For $\Gamma_s$, the only contractions are associated to scratch points at infinity $p = (c, q)$. The contractions are of the form $(C \times q) \times p$.

For $\Gamma_r$, the only contractions are associated to isotropic scratch points $p = (c, q)$. The contractions are of the form $(c \times D) \times p$. We have $s^{-1}( (c \times D) \smallsetminus p) = s( (c \times D) \smallsetminus p)$ by self-adjointness of $s$. Thus the remaining contractions of $\Gamma_b$ are $s (c \times D) \times p$.

Given a contraction $V_0 \times p$ of the type $V_0 = C \times q$, we claim that in Lemma \ref{lemma_sc_crit}, we may take $m = 2$, $V'_1 = E_p \colonequals \pi^{-1}(p)$, $V'_2 = \hat{r}(V'_0)$. Indeed by Lemma \ref{lemma_secant_at_infty} and Lemma \ref{lemma_reflect_infty} we have:
\begin{itemize}
    \item $\hat{s}$ is regular on $V'_0$,
    \item $E_p \subset \hat{s}(V'_0)$,
    \item $\hat{r}$ is regular on $E_p$,
    \item $V_1 = E_p \subset \hat{r}(E_p)$,
    \item $\hat{s}$ is regular on $E_p$,
    \item $V'_0 \subset \hat{s}(E_p)$,
    \item $\hat{r}$ is regular on $V'_0$.
\end{itemize}
Since $\pi(\hat{r}(V'_0))$ is a curve, we are done.

Given a contraction $V_0 \times p$ of the type $V_0 = s(c \times D)$, we claim that in Lemma \ref{lemma_sc_crit}, we may take $m = 2$, $V'_1 = E_p$, $V'_2 = (c \times D)_{\str}$. By self-adjointness we know that the strict transform of $V'_0$ by $\hat{s}$ contains $(c \times D)_{\str}$. Then by Lemma \ref{lemma_secant_at_isotropic} and Lemma \ref{lemma_reflect_isotropic} we have:
\begin{itemize}
    \item $\hat{r}$ is regular on $(c \times D)_{\str}$,
    \item $V_1 = E_p \subset \hat{r}((c \times D)_{\str})$,
    \item $\hat{s}$ is regular on $E_p$,
    \item $E_p \subset \hat{s}(E_p)$,
    \item $\hat{r}$ is regular on $E_p$,
    \item $(c \times D)_{\str} \subset \hat{r}(E_p)$.
\end{itemize}
Since $\pi((c \times D)_{\str}) = (c \times D)$ is a curve, we are done.
\end{proof}

\bibliographystyle{plain}
\bibliography{bib}
\end{document}